\documentclass[reqno,12pt,a4paper]{amsart}

\usepackage{tikz}

\usepackage[mathscr]{eucal}
\usepackage{graphics,epic}
\usepackage{amsfonts}
\usepackage{amscd}
\usepackage{latexsym}
\usepackage{amsmath,amssymb, amsthm, stmaryrd, bm, bbm}
\usepackage[all,2cell]{xy}
\usepackage{rotating}

\newtheorem{theorem}{Theorem}[section]
\newtheorem*{theorem*}{Theorem}

\newtheorem{lemma}[theorem]{Lemma}
\newtheorem{proposition}[theorem]{Proposition}
\newtheorem{corollary}[theorem]{Corollary}

\newtheorem*{conjecture*}{Conjecture}

\newtheorem*{question*}{Question}
\theoremstyle{remark}
\newtheorem{remark}[theorem]{Remark}
\newtheorem{example}[theorem]{Example}
\theoremstyle{definition}
\newtheorem{definition}[theorem]{Definition}

\newcommand{\ie}{{\em i.e.~}\ }

\newcommand{\opname}[1]{\operatorname{\mathsf{#1}}}

\newcommand{\grmod}{\opname{grmod}\nolimits}
\newcommand{\proj}{\opname{proj}\nolimits}

\newcommand{\rad}{\opname{rad}\nolimits}
\renewcommand{\top}{\opname{top}\nolimits}
\newcommand{\coh}{\opname{coh}\nolimits}

\newcommand{\ind}{\opname{ind}}

\newcommand{\rep}{\opname{rep}\nolimits}

\newcommand{\cok}{\opname{cok}\nolimits}

\renewcommand{\ker}{\opname{ker}\nolimits}

\newcommand{\per}{\opname{per}\nolimits}

\newcommand{\Hom}{\opname{Hom}}
\newcommand{\End}{\opname{End}}
\newcommand{\Aut}{\opname{Aut}}

\newcommand{\Ext}{\opname{Ext}}

\newcommand{\ten}{\otimes}

\newcommand{\ca}{{\mathcal A}}

\newcommand{\cc}{{\mathcal C}}
\newcommand{\cd}{{\mathcal D}}

\newcommand{\ch}{{\mathcal H}}

\newcommand{\cp}{{\mathcal P}}

\newcommand{\fk}{{\mathrm k}}

\numberwithin{equation}{section}

\begin{document}

\title[The derived Hall algebra of a graded gentle one-cycle algebra II]{On the derived Hall algebra of a graded gentle one-cycle algebra II: Serre presentation}

\author{Hui Chen, and Dong Yang}

\address{Hui Chen, School of Biomedical Engineering and Informatics, Nanjing Medical University, Nanjing 211166, P. R. China.}
\email{huichen@njmu.edu.cn}

\address{Dong Yang, School of Mathematics, Nanjing University, Nanjing 210093, P. R. China}
\email{yangdong@nju.edu.cn}
\date{\today}
\begin{abstract}
We present Serre-type generators and relations for the derived Hall algebra of a graded gentle one-cycle algebra when a certain integral invariant is not $0$ or $\pm 1$.
\end{abstract}
\maketitle

\section{Introduction}
\label{s:intro}

Gentle algebras were introduced in \cite{AssemHappel81,AssemSkowronski87} to study finite-dimensional algebras which are derived equivalent to path algebras of type $\mathbb{A}$ and type $\tilde{\mathbb{A}}$. They also appeared in the classification of derived-discrete algebras \cite{Vossieck01,FushimiXing26} and as 2-Calabi--Yau-tilted algebras \cite{AssemBruestleCharbonneauPlamondon10}. In the past years due to the work \cite{HaidenKatzarkovKontsevich17} on topological Fukaya categories of graded marked surfaces graded gentle algebras have attracted the attention of many mathematicians, see for example \cite{LekiliPolishchuk20,OpperPlamondonSchroll18,JinSchrollWang23,Opper25} and the references therein.

In this paper, we describe the derived Hall algebra of a graded gentle one-cycle algebra by providing generators and relations. The AG-invariant of a graded gentle one-cycle algebra $A$ is either $\{(p_1,p_1+d),(p_2,p_2-d)\}$ for some integers $p_1,p_2\geq 1$ and $d\geq 0$, exactly when $A$ has finite global dimension, or $\{(q,q-d),(0,d)\}$ for some integers $q\geq 1$ and $d\in\mathbb{Z}$, exactly when $A$ has infinite global dimension, see for example \cite[Lemma 6.2]{ChenYang25}. If $d\neq 0$, then the number of connected components of the Auslander--Reiten quiver of $\per(A)$ is $3|d|$, when $A$ has finite global dimension, or $|d|$, when $A$ has infinite global dimension. For integers $p_1,p_2\geq 1$, $d\geq 2$ and $r\in\mathbb{Z}$, let $\ch(p_1,p_2,d,r)$ be the $\mathbb{Q}$-algebra generated by $\{x_{i,j}\mid i\in\mathbb{Z},~j=0,1,\ldots,d-1\}$ subject to relations given in Section~\ref{s:derived-Hall-of-zigzag} after Theorem~\ref{thm-for-zigzag}. For integers $d\geq 2$ and $r\in\mathbb{Z}$, let $\ch(0,1,d,r)$ be the $\mathbb{Q}$-algebra generated by $\{x_{i,j}\mid i\in\mathbb{Z},~j=0,1,\ldots,d-1\}$ subject to relations given in Section~\ref{ss:derived-Hall-of-linear-bounded} after Theorem~\ref{thm-for-linear}.

\begin{theorem}
\label{main-theorem}
Let $A$ be a graded gentle one-cycle algebra and $\ch(\per(A))$ be the derived Hall algebra of $\per(A)$.
\begin{itemize}
\item[(a)] Assume that the AG-invariant of $A$ is $\{(p_1,p_1+d),(p_2,p_2-d)\}$ for some integers $p_1,p_2\geq 1$ and $d\geq 2$. Then $\ch(\per(A))$ is isomorphic to $\ch(p_1,p_2,d,-1)$.
\item[(b)] Assume that the AG-invariant of $A$ is $\{(q,q-d),(0,d)\}$ for some integers $q\geq 1$ and $|d|\geq 2$. Then $\ch(\per(A))$ is isomorphic to $\ch(0,1,|d|,\mathrm{sgn}(d)q)$.
\end{itemize}
\end{theorem}

We have several remarks. First, the case $d=0$ is exceptional, see Remark~\ref{rem:d=0}; the case $|d|=1$ is also exceptional and will be treated in a separate paper. Secondly, we also have a version of Theorem~\ref{main-theorem} (b) for $\cd_{fd}(A)$ (Theorem~\ref{thm:derived-hall-algebra} (c)), the finite-dimensional derived category of $A$, provided that $\cd_{fd}(A)$ is left locally homologically finite (this happens exactly when $d>0$). Thirdly, if $A=k[x]/(x^2)$ is the graded algebra of dual numbers (whose AG-invariant is $\{(1,|x|),(0,1-|x|)\}$), Theorem~\ref{main-theorem} (b) was given in \cite{KellerYangZhou09}; if $A$ is derived equivalent to a gentle one-cycle algebra concentrated in degree $0$, Theorem~\ref{main-theorem} (b) (including the case $|d|=1$) as well as its $\cd_{fd}(A)$-version was given in \cite{BobinskiSchmude20}. Finally, the key point of our proof is that $\per(A)$ is triangle equivalent to the twisted root category of a certain infinite quiver $Q$ of type $\mathbb{A}_\infty^\infty$, a result established in \cite{ChenYang25}. Therefore we study derived Hall algebras of general twisted root categories. This does not only provide us with standard generators for $\ch(\per(A))$, namely, the shifts of simple representations of $Q$, but also greatly simplifies our computation of relations. More precisely, we show that all relations involve only two shifted simple representations: if they are shifted by the same degree, then the relation is the corresponding relation in the Ringel--Hall algebra of the category of finite-dimensional representations of $Q$ twisted by a certain anti-symmetric orbital Euler form (therefore these relations are twisted quantum Serre relations of type $\mathbb{A}$), and if they are shifted by different degrees then the relation is a quantum commutative relation, related to the same anti-symmetric orbital Euler form. Therefore, the computation of relations is reduced to the computation of the anti-symmetric orbital Euler form.

The structure of the paper is as follows. In Section~\ref{s:quiver-generalised-zigzag-linear} we recall and study two quivers of type $\mathbb{A}_\infty^\infty$, including their Ringel--Hall algebras and values of certain orbital Euler forms for simple and indecomposable projective representations. In Section~\ref{s:twisted-root-category} we recall the definition of twisted root categories and study their basic properties and their derived Hall algebras. Based on the results in Section~\ref{s:twisted-root-category}, we provide generators and relations for derived Hall algebras of twisted root categories of quivers of type $\mathbb{A}_\infty^\infty$ with generalised zigzag orientation in Section~\ref{s:derived-Hall-of-zigzag} and of the quiver of type $\mathbb{A}_\infty^\infty$ with linear orientation in Section~\ref{s:derived-Hall-of-linear}. Section~\ref{s:derive-Hall-algebra-of-graded-gentle-one-cycle-algebra} is devoted to the proof of Theorem~\ref{main-theorem}  and the $\cd_{fd}$-version of Theorem~\ref{main-theorem} (b).

\smallskip
Throughout the paper, let $q$ be a prime power and $\fk=\mathbb{F}_q$ be the finite field with $q$ elements. For a $\fk$-category $\ca$ and an object $X$ of $\ca$, denote by $[X]$ the isoclass of $X$, by $\Aut(X)$ the automorphism group of $X$ and denote by $\cp(\ca)$ the set/class of isoclasses of objects of $\ca$, by $\ind(\ca)$ the set/class of isoclasses of indecomposable objects of $\ca$.  For a finite set $S$, denote by $|S|$ the cardinality of $S$. Let $\Sigma$ denote the shift functor of any triangulated category. For a rational number $x$, we denote by $\mathrm{sgn}(x)$ the sign of $x$.

\medskip
\noindent{\it Aknowledgement.} The authors thank Jie Xiao and Bangming Deng for their constant encouragement and they thank Haicheng Zhang for patiently answering their questions. The second-named author acknowledges support by National Key R\&D Program of China 2024YFA1013801.

\section{The quivers $Q_{p_1,p_2}$ and $Q^l$}
\label{s:quiver-generalised-zigzag-linear}
In this section, we recall the definition of the Ringel--Hall algebra of an abelian category and introduce the orbital Euler form associated with an automorphism. The main objects of study are the following two quivers of type $\mathbb{A}_\infty^\infty$:
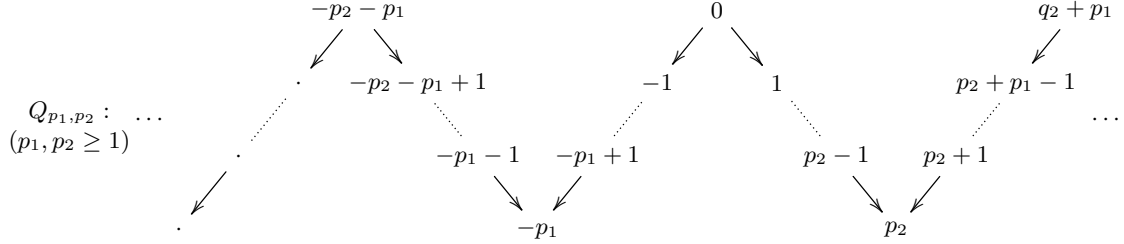
\[
{\scriptsize
\begin{xy} 0;<0.75pt,0pt>:<0pt,-0.55pt>::
(-25,70) *+{Q_{p_1,p_2}:}="", 
(-25,90) *+{(p_1,p_2\geq 1)}="", 
(15,75) *+{\cdots}="",
(30,150) *+{\cdot}="0",
(60,100) *+{\cdot}="1",
(90,50) *+{\cdot}="2",
(120,0) *+{-p_2-p_1}="3",
(150,50) *+{-p_2-p_1+1}="4",
(180,100) *+{-p_1-1}="5",
(210,150) *+{-p_1}="6",
(240,100) *+{-p_1+1}="7",
(270,50) *+{-1}="8",
(300,0) *+{0}="9",
(330,50) *+{1}="10",
(360,100) *+{p_2-1}="11",
(390,150) *+{p_2}="12",
(420,100) *+{p_2+1}="13",
(450,50) *+{p_2+p_1-1}="14",
(480,0) *+{q_2+p_1}="15",
(495,75) *+{\cdots}="",
"1", {\ar "0"}, "2", {\ar@{.} "1"}, "3", {\ar "2"},
"3", {\ar "4"}, "4", {\ar@{.} "5"}, "5", {\ar "6"},
"7", {\ar "6"}, "8", {\ar@{.} "7"}, "9", {\ar "8"},
"9", {\ar "10"}, "10", {\ar@{.} "11"}, "11", {\ar "12"},
"13", {\ar "12"}, "14", {\ar@{.} "13"}, "15", {\ar "14"},
\end{xy}
}
\]
with generalised zigzag orientation and with the quiver-automorphism $\sigma=\sigma_{p_1,p_2}$ which takes a vertex $i\in\mathbb{Z}$ to $i-p_1-p_2$, and
\[Q^l: \xymatrix{\ldots\ar[r]&i-1\ar[r]&i\ar[r]&i+1\ar[r]&\ldots}\]
with linear orientation and with quiver-automorphism $\sigma=\sigma_{0,1}$ which takes the vertex $i$ to $i-1$. Precisely, we recall the representation theory of the quivers $Q_{p_1,p_2}$ and $Q^l$ defined below, describe their Ringel--Hall algebras and study the values of the orbital Euler form associated with a power of $\sigma$ for simple and projective representations.

\subsection{Quivers and their representations}
\label{ss:quivers}

Let $Q$ be a (possibly infinite) quiver with vertex set $Q_0$ and arrow set $Q_1$. For an arrow $\alpha$, denote by $s(\alpha)$ the source of $\alpha$ and by $t(\alpha)$ the target of $\alpha$. Assume that $Q$ is \emph{strongly locally finite} in the sense of \cite[Section 1]{BautistaLiuPaquette13}, that is, for any $i\in Q_0$, the number of arrows with source $i$ is finite and the number of arrows with target $i$ is also finite; moreover, the number of paths between any two vertices is finite. This implies that $Q$ has no oriented cycles. For a representation $V$ of $Q$ (over $\fk$), let $V_i$ denote the vector space associated with a vertex $i$ and let $V_\alpha$ denote the $\fk$-linear map associated with an arrow $\alpha$. $V$ is said to be \emph{finite-dimensional} if $\bigoplus_{i\in Q_0} V_i$ is finite-dimensional over $\fk$. The \emph{simple representation} $S_i$ at a vertex $i$ is the representation with $(S_i)_i=\fk$, $(S_i)_j=0$ for $j\neq i$ and $(S_i)_\alpha=0$ for any $\alpha\in Q_1$.  Denote by $\rep^b(Q)$ the category of finite-dimensional representations of $Q$. This is a Hom-finite and Ext-finite skeletally small hereditary abelian $\fk$-category. For each $i\in Q_0$, we define a representation $P_i$ of $Q$ as follows. For $j\in Q_0$, the vector space $(P_i)_j$ is the vector space spanned by paths from $i$ to $j$; and for $\alpha\in Q_1$, the $k$-linear map $(P_i)_{\alpha}\colon (P_i)_{s(\alpha)}\to (P_i)_{t(\alpha)}$ takes a path $p$ to $\alpha p$. A representation $V$ of $Q$ is said to be \emph{finitely presented} if there is a short exact sequence $0\to P^{-1}\to P^0\to V\to 0$, where $P^{-1},~P^0\in\proj(Q)$, the category of finite direct sums of $P_i$'s, $i\in Q_0$. According to \cite[Proposition 1.15]{BautistaLiuPaquette13}, $\rep^+(Q)$ is a Hom-finite and Ext-finite hereditary abelian $k$-category. Moreover, $\rep^+(Q)$ contains $\rep^b(Q)$ as a full subcategory.

Define the \emph{Euler form}
\[
\langle V,W\rangle:=\dim \Hom_Q(V,W)-\dim \Ext^1_Q(V,W),
\]
where $V,W\in \rep^+(Q)$.
For two vertices $i$ and $j$, the dimension of $\Hom_Q(S_i,S_j)$ is the Kronecker symble $\delta_{ij}$ and the dimension of $\Ext^1_Q(S_i,S_j)$ equals the number of arrows from $i$ to $j$, so we have
\[
\langle S_i,S_j\rangle=\begin{cases} 1 & i=j\\ -|\text{arrows }i\to j| & i\neq j\end{cases}.
\]
For example, for the quiver of type $\mathbb{A}_n$ with linear orientation
\[\xymatrix{1\ar[r]&2\ar[r]&\ldots\ar[r]&n-1\ar[r]&n},\]
the values of the Euler form for simple representations are
\[\langle S_i,S_{j}\rangle=\begin{cases} 1 & \text{if } j=i \\
    -1 & \text{if } j=i+1 \\
    0 & \text{otherwise}
  \end{cases}.\]

Let $Q$ be a quiver of type $\mathbb{A}_\infty^\infty$, \ie its underlying graph is
\[\xymatrix{\ldots\ar@{-}[r]&i-1\ar@{-}[r]&i\ar@{-}[r]&i+1\ar@{-}[r]&\ldots}\]
For $a,b\in\{-\infty\}\cup\mathbb{Z}\cup\{+\infty\}$ with $a\leq b$, define a representation $V_{a,b}$ of $Q$ by
\begin{align*}(V_{a,b})_i&=\begin{cases} \fk & \text{if } a\leq i\leq b\\ 0 & \text{otherwise}\end{cases} \text{ for }i\in\mathbb{Z},\\
(V_{a,b})_\alpha&=\begin{cases} \mathrm{id}_{\fk} &\text{if } a\leq i,j\leq b\\ 0 & \text{otherwise}\end{cases} \text{ for an arrow }\alpha\colon i\to j.
\end{align*}
The endomorphism algebras of these representations are isomorphic to $\fk$ and hence they are indecomposable. They form a complete set of pairwise non-isomorphic indecomposable representations of $Q$, and $V_{a,b}$ ($a,b\in\mathbb{Z}$, $a\leq b$) form a complete set of pairwise non-isomorphic indecomposable finite-dimensional representations of $Q$. This can be obtained for example by the general theory in \cite[Section 5]{BautistaLiuPaquette13}.  Any finite-dimensional representation of $Q$ is supported on a finite connected subquiver, which is of type $\mathbb{A}_n$ for some $n$, and hence the above classification for finite-dimensional representations can also be obtained locally, see for example \cite{HouYe2006} and \cite{HouYe2008} for more details. Note that for $i\in\mathbb{Z}$ the simple representation $S_i$ is exactly $V_{i,i}$ and the projective representation $P_i$ is exactly $V_{i,+\infty}$. For the quiver $Q_{p_1,p_2}$ we have $\rep^+(Q_{p_1,p_2})=\rep^b(Q_{p_1,p_2})$, and for the quiver $Q^l$, the indecomposable finitely presented representations which are not finite-dimensional are precisely the $P_i$'s, $i\in\mathbb{Z}$.

Let $p_1,p_2\geq 1$ and $n=p_1+p_2$. For an integer $i\in \mathbb{Z}$, denote by $\overline{i}$ the unique integer in the interval $[-p_1,p_2)$ such that $n|(i-\overline{i})$.
In the following lemma we list the Hom-spaces, Ext-spaces and the Euler form between simple representations of $Q_{p_1,p_2}$.

\begin{lemma}
\label{lem:Euler-form-zigzag}
For $i,j\in\mathbb{Z}$, we have $\Hom_{Q_{p_1,p_2}}(S_i,S_j)=\begin{cases} \fk & j=i\\ 0 &j\neq i\end{cases}$; moreover,
\begin{itemize}
\item[(1)] if $\overline{i}=-p_1$, then
\[
\Ext^1_{Q_{p_1,p_2}}(S_i,S_j)=0~~\text{and}~~\langle S_i,S_j\rangle=\begin{cases} 1 & j=i\\ 0 &j\neq i\end{cases};
\]
\item[(2)] if $\overline{i}\in (-p_1,0)$, then
\[
\Ext^1_{Q_{p_1,p_2}}(S_i,S_j)=\begin{cases} \fk & j=i-1\\ 0 & j\neq i-1\end{cases}~~\text{and}~~\langle S_i,S_j\rangle=\begin{cases} 1 & j=i\\ -1 & j=i-1\\ 0 &\text{otherwise}\end{cases};
\]
\item[(3)] if $\overline{i}=0$, then
\[
\Ext^1_{Q_{p_1,p_2}}(S_i,S_j)=\begin{cases} \fk & j=i\pm1\\ 0 & j\neq i\pm1\end{cases}~~\text{and}~~\langle S_i,S_j\rangle=\begin{cases} 1 & j=i\\ -1 & j=i\pm1\\ 0 &\text{otherwise}\end{cases};
\]
\item[(4)] if $\overline{i}\in (0,p_2)$, then
\[
\Ext^1_{Q_{p_1,p_2}}(S_i,S_j)=\begin{cases} \fk & j=i+1\\ 0 & j\neq i+1\end{cases}~~\text{and}~~\langle S_i,S_j\rangle=\begin{cases} 1 & j=i\\ -1 & j=i+1\\ 0 &\text{otherwise}\end{cases}.
\]
\end{itemize}
\end{lemma}

In other words, the non-trivial Ext-spaces are $\Ext^1(S_i,S_{i-1})=\fk$ for any $i\in\mathbb{Z}$ with $\overline{i}\in (-p_1,0]$ and $\Ext^1(S_i,S_{i+1})=\fk$ for any $i\in\mathbb{Z}$ with $\overline{i}\in [0,p_2)$; the non-zero $\langle S_i,S_j\rangle$'s  are $\langle S_i,S_{i}\rangle=1$ for any $i\in\mathbb{Z}$, $\langle S_i,S_{i-1}\rangle=-1$ for any $i\in\mathbb{Z}$ with $\overline{i}\in (-p_1,0]$ and $\langle S_i,S_{i+1}\rangle=-1$ for any $i\in\mathbb{Z}$ with $\overline{i}\in [0,p_2)$.

\smallskip
Next we list the Hom-spaces, Ext-spaces and the Euler form between simple and projective representations of $Q^l$.

\begin{lemma}
	\label{lem:Euler-form-linear}
	For $i,j\in\mathbb{Z}$, we have
	
	\begin{itemize}
		\item[(1)] $\Hom_{Q^l}(S_i,S_j)=\begin{cases} \fk & j=i\\ 0 &j\neq i\end{cases},~~\Ext^1_{Q^l}(S_i,S_j)=\begin{cases} \fk & j=i+1\\ 0 & j\neq i+1\end{cases}~~\text{and}\\\langle S_i,S_j\rangle=\begin{cases} 1 & j=i\\ -1 & j=i+1\\ 0 &\text{otherwise}\end{cases};$
		\item[(2)] $
		\Hom_{Q^l}(P_i,P_j)=\begin{cases} \fk & j\leq i\\ 0 &j> i\end{cases},~~\Ext^1_{Q^l}(P_i,P_j)=0~~\text{and}~~\langle P_i,P_j\rangle=\begin{cases} 1 & j\leq i\\ 0 &j>i\end{cases};$
		\item[(3)] $	\Hom_{Q^l}(S_j,P_i)=0,~~\Ext^1_{Q^l}(S_j,P_i)=\begin{cases} \fk & j=i-1\\ 0 & j\neq i-1\end{cases}~~\text{and}~~\langle S_j, P_i\rangle=\begin{cases} -1 & j=i-1\\ 0 & j\neq i-1\end{cases};$
		\item[(4)] $
		\Hom_{Q^l}(P_i,S_j)=\begin{cases} \fk & j=i\\ 0 &j\neq i\end{cases},~~\Ext^1_{Q^l}(P_i,S_j)=0~~\text{and}~~\langle P_i,S_j\rangle=\begin{cases} 1 & j=i\\ 0 &j\neq i\end{cases}.$
		
	\end{itemize}
\end{lemma}

\subsection{Ringel--Hall algebras}\label{ss:Ringel-Hall-algebra}
Let $\mathcal{A}$ be a Hom-finite and Ext-finite skeletally small hereditary abelian $\fk$-category.
The \emph{Ringel--Hall algebra} $\mathcal{H(A)}$ of $\mathcal{A}$ is the  $\mathbb{Q}$-vector space with basis $\mathcal{P(A)}$, the set of isoclasses of objects of $\ca$, and for any $[M], [N]\in \mathcal{P(A)}$, the multiplication is defined by
\[[M]\diamond [N]=\sum_{[L]\in\cp(\ca)}g_{MN}^L [L],\]
where $g_{MN}^L:=|\{U\subseteq L\mid L/U\cong M, U\cong N\}|$, called the \emph{Hall number}. Note that if $g_{MN}\neq 0$, then $g_{MN}^L=\frac{|\Hom_\ca(N,L)_M|}{|\Aut_\ca(N)|}$, where $\Hom_\ca(N,L)_M$ is the set of morphisms $N\to L$ with cokernel $M$. By \cite{Ringel90}, $\mathcal{H(A)}$ is an associative $\mathbb{Z}$-algebra with unit the isoclass $[0]$ of the zero object.

\begin{remark}
For objects $L,M,N\in \mathcal{A}$, there is the Riedtmann--Peng formula \cite{Riedtmann94,Peng97},
\[g_{MN}^L=\frac{|\Ext_{\mathcal{A}}(M,N)_L|}{|\Hom_{\mathcal{A}}(M,N)|}\cdot \frac{|\Aut_\ca (L)|}{|\Aut_\ca(M)|\cdot|\Aut_\ca(N)|},\]
where $\Ext^{1}_{\mathcal{A}}(M,N)_L$ denote the subset of $\Ext^{1}_{\mathcal{A}}(M,N)$ which consists of the equivalence classes of short exact sequences with middle term $L$.
\end{remark}

\begin{example}
Let $Q=\xymatrix{1\ar[r]&2}$. Recall that $P_1$ is the projective cover of the simple representation $S_1$. In $\mathcal{H}(\rep^b(Q))$ we have
\begin{align*}
  {[S_1]\diamond [S_2]} &= {[P_1]+[S_1\oplus S_2]}, \\
  {[S_2]\diamond [S_1]} &= [S_1\oplus S_2],
\end{align*}
and further
\begin{align*}
[S_1]^{\diamond 2}\diamond [S_2] &= (q+1)[P_1\oplus S_1]+(q+1)[S_1^{\oplus 2}\oplus S_2],\\
[S_1]\diamond [S_2]\diamond [S_1] &= [P_1\oplus S_1] +(q+1)[S_1^{\oplus 2}\oplus S_2],\\
[S_2]\diamond [S_1]^{\diamond 2} &= (q+1)[S_1^{\oplus 2}\oplus S_2].
\end{align*}
Therefore
  \begin{equation*}
   [S_1]^{\diamond2}\diamond[S_{2}]-(1+q)[S_1]\diamond[S_{2}]\diamond[S_1]+q[S_{2}]\diamond[S_1]^{\diamond2}.
  \end{equation*}
Actually, let $\ch$ be the $\mathbb{Q}$-algebra generated by $\{x_1, x_2\}$ subject to relations
\begin{align*}
    x_1x_2^2-(1+q)x_2 x_1x_2+qx_2^2x_1,\\
   x_1^2 x_2-(1+q) x_1 x_2 x_1+qx_2x_1^2.
  \end{align*}
Then the assignment $x_i\mapsto [S_i]$ ($i=1,2$) extends to an isomorphism $\ch\to\ch(\rep^b(Q))$ of $\mathbb{Q}$-algebras.
\end{example}

Next we describe the structure of the Ringel--Hall algebras $\mathcal{H}(\rep^b(Q_{p_1,p_2}))$, $\mathcal{H}(\rep^b(Q^l))$ and $\mathcal{H}(\rep^+(Q^l))$, where $Q_{p_1,p_2}$ and $Q^l$ are the quivers introduced in the beginning of this section.

First, let $\ch^{l}$ be the $\mathbb{Q}$-algebra generated by $\{x_i\}_{i\in\mathbb{Z}}$, subject to relations
\begin{itemize}
\item for $i,j\in\mathbb{Z}$ with $|i-j|>1$: $x_ix_j-x_jx_i$,
\item for $i\in\mathbb{Z}$:
\begin{gather*}
x_ix_{i+1}^2-(1+q)x_{i+1}x_ix_{i+1}+qx_{i+1}^2x_i,\\
x_i^2x_{i+1}-(1+q)x_ix_{i+1}x_i+qx_{i+1}x_i^2.
\end{gather*}
\end{itemize}

We can view $Q^l$ as a limit (in both directions) of the quiver $\overrightarrow{\mathbb{A}}_n$ of type $\mathbb{A}_n$ with linear orientation and obtain the Ringel--Hall algebra $\mathcal{H}(\rep^b(Q^l))$ as a limit of the Ringel--Hall algebra $\mathcal{H}(\rep^b(\overrightarrow{\mathbb{A}}_n))$, see \cite{HouYe2006,HouYe2008}. In particular, we have

\begin{lemma}
\label{lem:RH-algebra-of-quiver-with-linear-orientation}
The assignment $x_i\mapsto [S_i]$ ($i\in\mathbb{Z}$) extends to an isomorphism $\ch^{l}\to\ch(\rep^b(Q^l))$ of $\mathbb{Q}$-algebras.
\end{lemma}

Secondly, let $\ch_{p_1,p_2}$ be the $\mathbb{Q}$-algebra generated by $\{x_i\}_{i\in\mathbb{Z}}$, subject to relations
\begin{itemize}
\item for $i,j\in\mathbb{Z}$ with $|i-j|>1$: $x_ix_j-x_jx_i$,
\item for $i\in\mathbb{Z}$ with $\overline{i}\in[-p_1,0)$:
\begin{gather*}
x_ix_{i+1}^2-(1+q^{-1})x_{i+1}x_ix_{i+1}+q^{-1}x_{i+1}^2x_i,\\
x_i^2x_{i+1}-(1+q^{-1})x_ix_{i+1}x_i+q^{-1}x_{i+1}x_i^2,
\end{gather*}
\item for $i\in\mathbb{Z}$ with $\overline{i}\in [0,p_2)$:
\begin{gather*}
x_ix_{i+1}^2-(1+q)x_{i+1}x_ix_{i+1}+qx_{i+1}^2x_i,\\
x_i^2x_{i+1}-(1+q)x_ix_{i+1}x_i+qx_{i+1}x_i^2.
\end{gather*}
\end{itemize}

We can view $Q^{p_1,p_2}$ as a limit of some quivers of type $\mathbb{A}_n$, and as a consequence we obtain, as for $\ch(\rep^b(Q^l))$ above, the following result for $\mathcal{H}(\rep^b(Q_{p_1,p_2}))$.

\begin{lemma}
\label{lem:RH-algebra-of-quiver-with-zigzag-orientation}
The assignment $x_i\mapsto [S_i]$ ($i\in\mathbb{Z}$) extends to an isomorphism $\ch_{p,q}\to\ch(\rep^b(Q_{p_1,p_2}))$ of $\mathbb{Q}$-algebras.
\end{lemma}

Finally, let $\ch^{l,+}$ be the $\mathbb{Q}$-algebra generated by $\{x_i\}_{i\in\mathbb{Z}}\cup \{z_i\}_{i\in\mathbb{Z}}$ such that $\{x_i\}_{i\in\mathbb{Z}}$ generate $\ch^{l}$ as a subalgebra of $\ch^{l,+}$ and subject to the following additional relations:
\begin{itemize}
	\item for $i,j\in\mathbb{Z}$:
	\begin{gather*}
		 x_jz_i-z_ix_j ~\hbox{if $|i-j|>1$},\\
		 z_iz_j-q^{\mathrm{sgn}(j-i)}z_jz_i, 
	\end{gather*}
	\item for $i\in\mathbb{Z}$:
	\begin{gather*}
		x_iz_i-qz_ix_i,\\
		x_{i-1}z_{i}-z_{i}x_{i-1}-z_{i-1}.
	\end{gather*}
\end{itemize}

\begin{lemma}
	\label{lem:RH-algebra-of-quiver-with-linear-orientation-$+$}
	The assignment $x_i\mapsto [S_i]$ ($i\in\mathbb{Z}$) and $z_i\mapsto [P_i]$ ($i\in\mathbb{Z}$) extends to an isomorphism $\ch^{l,+}\to\ch(\rep^+(Q^l))$ of $\mathbb{Q}$-algebras.
\end{lemma}
\begin{proof}
First, it is straightforward to verify that the above relations are satisfied in $\ch(\rep^+(Q^l))$ with $x_i$ replaced by $[S_i]$ and $z_i$ replaced by $[P_i]$. Therefore the assignment $x_i\mapsto [S_i]$ ($i\in\mathbb{Z}$) and $z_i\mapsto [P_i]$ ($i\in\mathbb{Z}$) extends to a homomorphism $\varphi\colon\ch^{l,+}\to\ch(\rep^+(Q^l))$ of $\mathbb{Q}$-algebras. Secondly, let $\mathcal{H}^{l,0}$ be the subalgebra of $\mathcal{H}^{l,+}$ generated by $z_i$ ($i\in\mathbb{Z}$). Then $\ch^{l,0}$ has a basis $z_{i_1}^{a_{i_1}}\cdots z_{i_k}^{a_{i_k}}$, where $i_1,\ldots,i_k\in\mathbb{Z}$ with $i_1<\ldots<i_k$, and $a_{i_1},\ldots,a_{i_k}\in\mathbb{N}$, and $\ch^{l,+}=\ch^{l,0}\ten_{\mathbb{Q}}\ch^l$.  Thirdly, any representation $\rep^+(Q^l)$ is, up to isomorphism, of the form $P\oplus M$ with $P\in \proj(Q^l)$ and $M\in\rep^b(Q^l)$, and we have $[P]\diamond[M]=[P\oplus M]$, since $P$ is projective and $\Hom_{Q^l}(P,M)=0$. This shows that $\ch(\rep^+(Q^l))=\ch(\proj(Q^l))\ten_{\mathbb{Q}}\ch(\rep^b(Q^l))$, where $\ch(\proj(Q^l))$ is the subalgebra of $\ch(\rep^+(Q^l))$ with basis $\cp(\proj Q^l)$. Fourthly, by Lemma~\ref{lem:RH-algebra-of-quiver-with-linear-orientation}, $\varphi$ restricts to an isomorphism $\ch^l\to\ch(\rep^b(Q^l))$. Therefore it remains to show that $\varphi$ restricts to an isomorphism $\ch^{l,0}\to\ch(\proj(Q^l))$ of $\mathbb{Q}$-algebras. Observe that $\varphi$ sends $z_{i_1}^{a_{i_1}}\cdots z_{i_k}^{a_{i_k}}$ to $[P_{i_1}]^{\diamond a_{i_1}}\diamond\cdots\diamond [P_{i_k}]^{\diamond a_{i_k}}$, which is a scalar multiple of $[P_{i_1}^{\oplus a_{i_1}}\oplus\cdots \oplus P_{i_k}^{\oplus a_{i_k}}]$. The proof is then finished, because $[P_{i_1}^{\oplus a_{i_1}}\oplus\cdots \oplus P_{i_k}^{\oplus a_{i_k}}]$, $i_1,\ldots,i_k\in\mathbb{Z}$ with $i_1<\ldots<i_k$ and $a_{i_1},\ldots,a_{i_k}\in\mathbb{N}$, form a basis of $\ch(\proj(Q^l))$. 
\end{proof}

\subsection{Orbital Euler form}\label{ss:orbit-Euler-form}
Let $\mathcal{A}$ be a Hom-finite and Ext-finite skeletally small hereditary abelian $\fk$-category. The \emph{Euler form} of $\ca$ is defined to be:
\[\langle M,N\rangle:={\rm dim}_{\fk}{\rm Hom}_{\mathcal{A}}(M,N)-{\rm dim}_{\fk}{\rm Ext}^1_{\mathcal{A}}(M,N),\]
where $M,N\in \mathcal{A}$.

Assume that $\sigma$ is a $\fk$-linear automorphism of $\mathcal{A}$ of infinite order satisfying the condition
\begin{itemize}
\item[(HF)]
for any $M,N\in\ca$ the spaces $\Hom_\ca(M,\sigma^{-p}N)$ and $\Ext^1_\ca(M,\sigma^{-p}N)$ vanish for all but finitely many positive integers $p$.
\end{itemize}

\begin{definition}
For an integer $d\geq 1$, the \emph{$d$-twisted left $\sigma$-orbital Euler form} is defined to be:
\[\langle M,N\rangle_o:=\sum_{p\in \mathbb{Z}^+}\langle M,\sigma^{-p}N\rangle\cdot(-1)^{dp},\]
where $M,N\in\ca$. The corresponding anti-symmetric form is:
\[ (M,N)_o:=\langle M,N\rangle_o-\langle N,M\rangle_o.\]
\end{definition}

\begin{lemma}
\label{lem:rotate-orbital-Euler-form}
For $M,N\in\ca$ and $1\leq t\leq d-1$, we have
\begin{align*}
\langle \sigma^{-1}M,N\rangle_o=\langle M,\sigma N\rangle_o&=(-1)^d \langle M,N\rangle+(-1)^d\langle M,N\rangle_o,\\
(-1)^{d-t}\langle \sigma^{-1}N,M\rangle_o-(-1)^{t}\langle M,N\rangle_o&=(-1)^{t-1}((M,N)_o-\langle N,M\rangle).
\end{align*}
\end{lemma}
\begin{proof}
The first equality holds because $\sigma$ is an automorphism. The second equality is straightforward:
\begin{align*}
\langle M,\sigma N\rangle_o&=\sum_{p\in \mathbb{Z}^+}\langle M,\sigma^{-p+1}N\rangle\cdot(-1)^{dp}=\sum_{p\in \mathbb{Z}^+\cup\{0\}}\langle M,\sigma^{-p}N\rangle\cdot(-1)^{d(p+1)}\\
&=(-1)^d\langle M,N\rangle+\sum_{p\in \mathbb{Z}^+}\langle M,\sigma^{-p}N\rangle\cdot(-1)^{d(p+1)}\\
&=(-1)^d \langle M,N\rangle+(-1)^d\langle M,N\rangle_o.
\end{align*}
For the third equality:
\begin{align*}
(-1)^{d-t}\langle \sigma^{-1}N,M\rangle_o-(-1)^{t}\langle M,N\rangle_o&=(-1)^{t-1}(\langle M,N\rangle_o-(-1)^{d}\langle \sigma^{-1}N,M\rangle_o)\\
&=(-1)^{t-1}(\langle M,N\rangle_o-\langle N,M\rangle_o-\langle N,M\rangle)\\
&=(-1)^{t-1}((M,N)_o-\langle N,M\rangle).\qedhere
\end{align*}
\end{proof}

\subsubsection{The quiver $Q_{p_1,p_2}$}
\label{ss:orbit-Euler-form-zigzag}

We consider the quiver $Q_{p_1,p_2}$ ($p_1,p_2\geq 1$) with quiver-automorphism $\sigma$ introduced in the beginning of this section. Recall that $\sigma$ takes a vertex $i\in\mathbb{Z}$ to $i-n$, where $n=p_1+p_2$. The corresponding push-out functor on $\rep^b(Q_{p_1,p_2})$, also denoted by $\sigma$, satisfies the condition (HF). Next we list the values of the orbital Euler form of $\rep^b(Q_{p_1,p_2})$ for simple representations.

\begin{lemma}
\label{lem:orbital-Euler-form-zigzag}
Let $d$ and $r$ be integers with $d\geq 1$ and $r\neq 0$.
The values of the $d$-twisted left $\sigma^r$-orbital Euler form for simple representations of $Q_{p_1,p_2}$ is
\begin{itemize}
\item[(1)] If $|nr|=2$ (that is, $p_1=p_2=1$ and $|r|=1$), then
    \[
    \langle S_i,S_j\rangle_o=\begin{cases}
    (-1)^{dk} & \text{if }j=i-2kr~(k>0)\\
    (-1)^{d+1}-1 & \text{if } \overline{i}=0,~j=i-(2k+1)r~(k>0)\\
    (-1)^{d+1}& \text{if } \overline{i}=0,~j=i-r\\
    0 & \text{otherwise}.
    \end{cases}
    \]
\item[(2)] If $|nr|>2$, then
    \[
    \langle S_i,S_j\rangle_o=\begin{cases}
    (-1)^{dk} & \text{if }j=i-nkr~(k>0)\\
    (-1)^{dk+1} & \text{if } \overline{i}\in(-p_1,0],~j=i-(nkr+1)~(k>0)\\
    (-1)^{dk+1} & \text{if } \overline{i}\in [0,p_2),~j=i-(nkr-1)~(k>0)\\
    0 & \text{otherwise}.
    \end{cases}
    \]
\end{itemize}
\end{lemma}
\begin{proof}
We give a proof when $nr=2$ (\ie $p_1=p_2=1$ and $r=1$) using Lemma~\ref{lem:Euler-form-zigzag}, and the proof of the other cases is similar.

When $j=i-2k$ ($k\in\mathbb{Z}$),
\begin{align*}
\langle S_i,S_j\rangle_o &=\sum_{p\in \mathbb{Z}^+}\langle S_i,\sigma^{-rp}S_{i-2k}\rangle\cdot(-1)^{dp}=\sum_{p\in \mathbb{Z}^+}\langle S_i,\sigma^{-p}S_{i-2k}\rangle\cdot(-1)^{dp}\\
&=\sum_{p\in \mathbb{Z}^+}\langle S_i,S_{i-2k+2p}\rangle\cdot(-1)^{dp}=
\begin{cases}
    \langle S_i,S_{i}\rangle\cdot(-1)^{dk} & \text{if }k>0\\
    0 & \text{otherwise}.
    \end{cases}\\
&=
\begin{cases}
    (-1)^{dk} & \text{if }j=i-2k ~(k>0)\\
    0 & \text{otherwise}.
    \end{cases}
\end{align*}
When $j=i-(2k+1)$ ($k \in \mathbb{Z}$), we need to discuss $\overline{i}=-1$ or $\overline{i}=0$. Without loss of generality, we may assume $i=-1$ or $i=0$.

Case $i=-1$: We have
  \begin{align*}
  \langle S_i,S_j\rangle_o &=\sum_{p\in \mathbb{Z}^+}\langle S_{-1},\sigma^{-rp}S_{-2-2k}\rangle\cdot(-1)^{dp}=\sum_{p\in \mathbb{Z}^+}\langle S_{-1},\sigma^{-p}S_{-2-2k}\rangle\cdot(-1)^{dp}\\
  &=\sum_{p\in \mathbb{Z}^+}\langle S_{-1},S_{-2-2k+2p}\rangle\cdot(-1)^{dp}=0.
  \end{align*}

Case $i=0$: We have
 \begin{align*}
  \langle S_i,S_j\rangle_o &=\sum_{p\in \mathbb{Z}^+}\langle S_0,\sigma^{-rp}S_{-1-2k}\rangle\cdot(-1)^{dp}=\sum_{p\in \mathbb{Z}^+}\langle S_0,\sigma^{-p}S_{-1-2k}\rangle\cdot(-1)^{dp}\\
&=\sum_{p\in \mathbb{Z}^+}\langle S_0,S_{-1-2k+2p}\rangle\cdot(-1)^{dp}\\
  &=\begin{cases}
    \langle S_0,S_{-1}\rangle\cdot(-1)^{dk}+\langle S_0,S_{1}\rangle\cdot(-1)^{d(k+1)} & \text{if }k>0\\
    \langle S_0,S_{1}\rangle\cdot(-1)^{d}& \text{if }k=0\\
    0 & \text{if }k<0
    \end{cases}\\
      &=
  \begin{cases}
    (-1)^{dk}[(-1)^{d+1}-1] & \text{if }j=i-(2k+1)~(k>0)\\
    (-1)^{d+1}& \text{if }j=i-1\\
    0 & \text{otherwise}.
    \end{cases}\\
  &=
  \begin{cases}
    (-1)^{d+1}-1 & \text{if }j=i-(2k+1)~(k>0)\\
    (-1)^{d+1}& \text{if }j=i-1\\
    0 & \text{otherwise}.
    \end{cases}
\end{align*}
The last equality holds because $(-1)^d[(-1)^{d+1}-1]=(-1)^{d+1}-1$.
\end{proof}

The following is an immediate corollary of Lemma~\ref{lem:orbital-Euler-form-zigzag}.
\begin{corollary}
\label{cor:antisymmetric-orbital-Euler-form-zigzag}
Let $d$ and $r$ be integers with $d\geq 1$ and $r\neq 0$.
The values of $(-,-)_o$ for simple representations of $Q_{p_1,p_2}$ is
\begin{itemize}
\item[(1)] If $|nr|=2$ (that is, $p_1=p_2=1$ and $|r|=1$), then
    \[
    (S_i,S_j)_o=\begin{cases}
    (-1)^{dk} & \text{if }j=i-2kr~(k>0)\\
    (-1)^{d+1}-1 & \text{if } \overline{i}=0,~j=i-(2k+1)r~(k>0)\\
    (-1)^{d+1}& \text{if } \overline{i}=0,~j=i-r\\
    (-1)^{dk+1} & \text{if }j=i+2kr~(k>0)\\
    (-1)^d+1 & \text{if } \overline{i}=-1,~j=i+(2k+1)r~(k>0)\\
    (-1)^{d}& \text{if } \overline{i}=-1,~j=i+r\\
    0 & \text{otherwise}.
    \end{cases}
    \]
\item[(2)] If $|nr|>2$, then
    \[
    (S_i,S_j)_o=\begin{cases}
    (-1)^{dk} & \text{if }j=i-nkr~(k>0)\\
    (-1)^{dk+1} & \text{if } \overline{i}\in(-p_1,0],~j=i-(nkr+1)~(k>0)\\
    (-1)^{dk+1} & \text{if } \overline{i}\in [0,p_2),~j=i-(nkr-1)~(k>0)\\
    (-1)^{dk+1} & \text{if }j=i+nkr~(k>0)\\
    (-1)^{dk} & \text{if } \overline{i}\in[-p_1,0),~j=i+(nkr+1)~(k>0)\\
    (-1)^{dk} & \text{if } \overline{i}\in (1,p_2],~j=i+(nkr-1)~(k>0)\\
    0 & \text{otherwise} .
    \end{cases}
    \]
\end{itemize}
\end{corollary}
In the second to the last row in (2) by $\overline{i}=p_2$ we actually mean $\overline{i}=-p_1$.

\subsubsection{The quiver $Q^l$}
\label{ss:orbit-Euler-form-linear}

We consider the quiver $Q^l$ with quiver-automorphism $\sigma$ introduced in the beginning of this section. Recall that $\sigma$ takes the vertex $i$ to $i-1$. We also denote by $\sigma$ the corresponding push-out functor on $\rep^+(Q^l)$, which restricts to an automorphism of $\rep^b(Q^l)$. For $r\in\mathbb{Z}$, $\sigma^r$ as an automorphism of $\rep^b(Q^l)$ always satisfies the condition (HF), and $\sigma^r$ as an automorphism of $\rep^+(Q^l)$ satisfies the condition (HF) if and only if $r>0$. Next we list the values of the orbital Euler form for simple representations and indecomposable projective representations.

\begin{lemma}
\label{lem:orbital-Euler-form-linear}
Let $d$ and $r$ be integers with $d\geq 1$ and $r\neq 0$.
The values of the $d$-twisted left $\sigma^r$-orbital Euler form for simple representations of $Q^l$ is
\begin{itemize}
\item[(1)] If $r=-1$, then
\[
    \langle S_i,S_j\rangle_o=\begin{cases}
    (-1)^{k}[(-1)^{d+1}+1] & \text{if }j=i+k~(k>1)\\
    (-1)^{d} & \text{if } j=i+1\\
    0 & \text{otherwise}.
    \end{cases}
    \]

\item[(2)] If $r=1$, then
\[
    \langle S_i,S_j\rangle_o=\begin{cases}
    (-1)^{k}[(-1)^{d+1}+1] & \text{if }j=i-k~(k>0)\\
    (-1)^{d+1} & \text{if } j=i\\
    0 & \text{otherwise}.
    \end{cases}
    \]

\item[(3)] If $|r|>1$, then
\[
    \langle S_i,S_j\rangle_o=\begin{cases}
    (-1)^{dk} & \text{if }j=i-kr~(k>0)\\
    (-1)^{dk+1} & \text{if } j=i-(kr-1)~(k>0)\\
    0 & \text{otherwise}.
    \end{cases}
    \]
\end{itemize}
\end{lemma}
\begin{proof}
We give a proof when $r=-1$ using Lemma~\ref{lem:Euler-form-linear}, and the proof of the other cases is similar.
For any $i, j \in \mathbb{Z}$, denote $k=j-i$. Then
\begin{align*}
\langle S_i,S_j\rangle_o &=\langle S_{i},S_{i+k}\rangle_o=\sum_{p\in \mathbb{Z}^+}\langle S_i,\sigma^{-rp}S_{i+k}\rangle\cdot(-1)^{dp}=\sum_{p\in \mathbb{Z}^+}\langle S_i,\sigma^{p}S_{i+k}\rangle\cdot(-1)^{dp}\\
&=\sum_{p\in \mathbb{Z}^+}\langle S_i,S_{i+k-p}\rangle\cdot(-1)^{dp}\\
&=
\begin{cases}
    \langle S_i,S_{i}\rangle\cdot(-1)^{dk} +\langle S_i,S_{i+1}\rangle\cdot(-1)^{d(k-1)} & \text{if }k>1\\
    \langle S_i,S_{i}\rangle\cdot(-1)^{dk} & \text{if } k=1\\
    0 & \text{if } k\leq 0
    \end{cases}\\
    &=
\begin{cases}
    (-1)^{dk}[(-1)^{d+1}+1] & \text{if }j=i+k ~(k>1)\\
    (-1)^{d} & \text{if } j=i+1\\
    0 & \text{otherwise.}
    \end{cases}\\
&=
\begin{cases}
    (-1)^{k}[(-1)^{d+1}+1] & \text{if }j=i+k ~(k>1)\\
    (-1)^{d} & \text{if } j=i+1\\
    0 & \text{otherwise.}
    \end{cases}
\qedhere
\end{align*}
\end{proof}

\begin{lemma}
	\label{lem:orbital-Euler-form-linear-projectives}
	Let $d$ and $r$ be integers with $d\geq 1$ and $r>0$.
	The values of the $d$-twisted left $\sigma^r$-orbital Euler form for indecomposalbe projective and simple representations of $Q^l$ is
	\begin{itemize}
		\item[(1)]
		$\langle P_i,P_j\rangle_o=\begin{cases} \sum\limits_{p=1}^{k}(-1)^{dp} &  kr\leq i-j<(k+1)r~(k>0)\\ 0 &\text{otherwise}\end{cases}.$ 
		\item[(2)] $
		\langle P_i,S_j\rangle_o=\begin{cases}
			(-1)^{dk} & \text{if }j=i-kr~(k>0)\\
			0 & \text{otherwise}
		\end{cases}.
		$
		\item[(3)]$
		\langle S_j,P_i\rangle_o=\begin{cases}
			(-1)^{dk+1} & \text{if }j=i+kr-1~(k>0)\\
			0 & \text{otherwise}
		\end{cases}$.
	\end{itemize}
\end{lemma}
\begin{proof}
	We give a proof for the case (1), and the proof of the other cases is similar.
	$\forall i, j \in \mathbb{Z}$, 
	\begin{align*}
		\langle P_i,P_j\rangle_o &=\sum_{p\in \mathbb{Z}^+}\langle P_i,\sigma^{-rp}P_{j}\rangle\cdot(-1)^{dp}=\sum_{p\in \mathbb{Z}^+}\langle P_i,P_{j+rp}\rangle\cdot(-1)^{dp}\\
		&=\begin{cases} \sum\limits_{p=1}^{k}(-1)^{dp} & k\leq \frac{i-j}{r}<(k+1)r~(k>0)\\ 0 &\text{otherwise}\end{cases}.
		\qedhere
	\end{align*}
\end{proof} 

The following are immediate corollaries of Lemma~\ref{lem:orbital-Euler-form-linear} and Lemma~\ref{lem:orbital-Euler-form-linear-projectives}. 

\begin{corollary}
\label{cor:antisymmetric-orbital-Euler-form-linear}
Let $d$ and $r$ be integers with $d\geq 1$ and $r\neq 0$.
The values of $(-,-)_o$ for simple representations of $Q^{l}$ is
\begin{itemize}
\item[(1)] If $r=-1$, then
    \[
    (S_i,S_j)_o=\begin{cases}
    (-1)^{k}[(-1)^{d+1}+1] & \text{if }j=i+k~(k>1)\\
    (-1)^{d} & \text{if } j=i+1\\
    (-1)^{k+1}[(-1)^{d+1}+1] & \text{if }j=i-k~(k>1)\\
    (-1)^{d+1} & \text{if } j=i-1\\
    0 & \text{if } j=i.
    \end{cases}
    \]
\item[(2)] If $r=1$, then
    \[
    (S_i,S_j)_o=\begin{cases}
    (-1)^{k}[(-1)^{d+1}+1] & \text{if }j=i-k~(k>0)\\
    (-1)^{k+1}[(-1)^{d+1}+1] & \text{if }j=i+k~(k>0)\\
    0 & \text{if } j=i.
    \end{cases}
    \]
    \item[(3)] If $|r|>1$, then
\[
    (S_i,S_j)_o=\begin{cases}
    (-1)^{dk} & \text{if }j=i-kr~(k>0)\\
    (-1)^{dk+1} & \text{if } j=i-(kr-1)~(k>0)\\
    (-1)^{dk+1} & \text{if }j=i+kr~(k>0)\\
    (-1)^{dk} & \text{if } j=i+(kr-1)~(k>0)\\
    0 & \text{otherwise}.
    \end{cases}
    \]
\end{itemize}
\end{corollary}

\begin{corollary}
	\label{cor:antisymmetric-orbital-Euler-form-linear-projectives}
	Let $d$ and $r$ be integers with $d\geq 1$ and $r > 0$.
	The values of $(-,-)_o$ for indecomposalbe projective and simple representations of $Q^{l}$ is
	\begin{itemize}
		\item[(1)] $
		(P_i,P_j)_o=\begin{cases}
			\sum\limits_{p=1}^{k}(-1)^{dp} & kr\leq i-j<(k+1)r~(k>0)\\
			\sum\limits_{p=1}^{k}(-1)^{dp+1} & -(k+1)r<i-j \leq -kr~(k>0)\\ 0 &\text{otherwise}\end{cases}.$
		\item[(2)] $
		(P_i,S_j)_o=\begin{cases}
			(-1)^{dk} & \text{if }j=i-kr~(k>0)\\
			(-1)^{dk} & \text{if }j=i+kr-1~(k>0)\\
			0 & \text{otherwise}
		\end{cases}$.
	\end{itemize}
\end{corollary}

\section{Twisted root categories and Hall algebras}
\label{s:twisted-root-category}
In this section we recall the definition of twisted root categories of a Hom-finite and Ext-finite hereditary abelian categories and study the basic properties and derived Hall algebras of them.
\subsection{The derived Hall algebra of a triangulated category}
\label{ss:derived-Hall-algebra-of-triangulated-cat}
Let $\mathcal{C}$ be a Hom-finite Krull--Schmidt triangulated $\fk$-category. Assume that $\cc$ is \emph{left locally homologically finite}, that is,  for any $X, Y\in \mathcal{C}$, the space $\Hom_\cc(X,\Sigma^{-l}Y)$ vanishes for all but finitely many positive integers $l$.

By \cite{Toen06,XiaoXu06}, for any $X,Y,Z\in \mathcal{C}$, the \emph{derived Hall number} is defined as\footnote{Note that this is opposite to the one in \cite{Toen06} and \cite{XiaoXu06}.}
\[F_{X,Y}^Z:=\frac{|\Hom_\cc(Y,Z)_X|}{|\Aut_\cc(Y)|}\cdot \frac{\{Y,Z\}}{\{Y,Y\}},\]
where $\Hom_\cc(Y,Z)_X$ is the set of morphisms from $Y$ to $Z$ with cone isomorphic to $X$ and $\{Y,Z\}=\Pi_{l>0}|{\rm Hom}_\mathcal{C}(Y,\Sigma^{-l}Z)|^{(-1)^l}$. The above formula is called  To\"{e}n's formula. There is also the derived Riedtmann--Peng formula \cite{XiaoXu15}:
\[
F_{X,Y}^Z=|\Hom_\cc(\Sigma^{-1}X,Y)_Z|\cdot\{\Sigma^{-1}X,Y\}\cdot\frac{|\Aut_\cc(Z)|}{|\Aut_\cc(X)|\cdot|\Aut_\cc(Y)|}\cdot\frac{\{Z,Z\}}{\{X,X\}\cdot\{Y,Y\}}.
\]

\begin{lemma}
\label{lemma-derived-Hall-number-for-direct-sum}
Let $X,Y,Z\in \mathcal{C}$ and $t\in\mathbb{Z}^+$. Then
\begin{align*}
\{X\oplus Y,Z\}&=\{X,Z\}\cdot \{Y,Z\},\\
\{Z,X\oplus Y\}&=\{Z,X\}\cdot \{Z,Y\},\\
\{X,\Sigma^t Y\}&=\{X,Y\}^{(-1)^t}\cdot(\Pi_{l=-t+1}^0|\Hom_\cc(X,\Sigma^{-l}Y)|^{(-1)^l})^{(-1)^t},\\
\{X,Y\}&=\{\Sigma^t X,Y\}^{(-1)^t}\cdot \Pi_{l=1}^{t}|\Hom_\cc(X,\Sigma^{-l}Y)|^{(-1)^l}.
\end{align*}
\end{lemma}
\begin{proof}
The first two equalities are easy. For the last two we have
\begin{align*}
\{X,\Sigma^t Y\}&=\Pi_{l>0}|\Hom_\cc(X,\Sigma^{-l}\Sigma^t Y)|^{(-1)^l}\\
&=\Pi_{l>0}|\Hom_\cc(X,\Sigma^{-(l-t)}Y)|^{(-1)^l}\\
&=\Pi_{l>-t}|\Hom_\cc(X,\Sigma^{-l}Y)|^{(-1)^{l+t}}\\
&=(\{X,Y\}\cdot\Pi_{l=-t+1}^0|\Hom_\cc(X,\Sigma^{-l}Y)|^{(-1)^l})^{(-1)^t},\\
\{X,Y\}&=\Pi_{l=1}^t|\Hom_\cc(X,\Sigma^{-l}Y)|^{(-1)^l}\cdot\Pi_{l>t}|\Hom_\cc(X,\Sigma^{-l}Y)|^{(-1)^l}\\
&=\Pi_{l=1}^t|\Hom_\cc(X,\Sigma^{-l}Y)|^{(-1)^l}\cdot\Pi_{l>t}|\Hom_\cc(\Sigma^t X,\Sigma^{-(l-t)}Y)|^{(-1)^l}\\
&=\Pi_{l=1}^t|\Hom_\cc(X,\Sigma^{-l}Y)|^{(-1)^l}\cdot\Pi_{l>0}|\Hom_\cc(\Sigma^t X,\Sigma^{-l}Y)|^{(-1)^{l+t}}\\
&=\{\Sigma^t X,Y\}^{(-1)^t}\cdot \Pi_{l=1}^{t}|\Hom_\cc(X,\Sigma^{-l}Y)|^{(-1)^l}.\qedhere
\end{align*}
\end{proof}

The \emph{derived Hall algebra} $\mathcal{H(C)}$ of $\mathcal{C}$ is the $\mathbb{Q}$-vector space with basis $\mathcal{P(C)}$, and the multiplication is defined by
\[[X][Y]=\sum_{[Z]\in\cp(\cc)}F_{X,Y}^Z [Z],~ \forall [X], [Y]\in \mathcal{P(C)}.\]
It is shown in \cite{Toen06,XiaoXu06} that $\mathcal{H(C)}$ is an associative algebra with unit $[0]$. The shift functor $\Sigma$ induces an algebra automorphism of $\ch(\cc)$, which we still denote by $\Sigma$.

\medskip
Following \cite[Definition 2.2]{Jorgensen22} we say that an additive subcategory $\ca$ of $\cc$ is a \emph{proper abelian subcategory} of $\cc$, if $\ca$ is abelian and $\xymatrix@C=0.7pc{0\ar[r] & N\ar[r]^f & L\ar[r]^g & M\ar[r] & 0}$ is a short exact sequence in $\ca$ if and only if $L, M,N\in\ca$ and there exists a triangle $\xymatrix@C=0.7pc{N\ar[r]^f & L\ar[r]^g & M\ar[r] & \Sigma N}$ in $\cc$. Typical examples are the heart of a $t$-structure and more generally the extension closure of a $w$-simple-minded system for $w\geq 2$, see \cite{Jorgensen22}.

\begin{lemma}
\label{lem:derived-Hall-number-for-proper-abelian-subcategory}
Let $\ca$ be an extension-closed proper abelian subcategory of $\cc$. 
Then for $M,N\in\ca$, the form $\{M,N\}$ is multiplicative in both $M$ and $N$. Moreover, we have
\[
[M][N]=\{N,M\}[M]\diamond [N].
\]
\end{lemma}
\begin{proof}
First, since $\ca$ is a proper abelian subcategory, it follows that if $L\in\ca$, then $\Hom_\cc(N,L)_M=\Hom_\ca(N,L)_M$, and therefore
\[
F_{MN}^L=\frac{|\Hom_\cc(N,L)_M|}{|\Aut_\cc(N)|}\frac{\{N,L\}}{\{N,N\}}=\frac{|\Hom_\ca(N,L)_M|}{|\Aut_\ca(N)|}\frac{\{N,L\}}{\{N,N\}}=g_{MN}^L\frac{\{N,L\}}{\{N,N\}}.
\]
Secondly, since $\ca$ is extension-closed, it follows that $F_{MN}^L=0$ if $L\not\in\ca$. So
\[
[M][N]=\sum_{[L]\in\cp(\ca)}\frac{\{N,L\}}{\{N,N\}}g_{MN}^L[L].
\]
Finally, assume that $\xymatrix@C=0.7pc{N\ar[r]^f & L\ar[r]^g & M\ar[r] & \Sigma N}$ is a triangle in $\cc$ with $L\in\ca$. Then the map $f_*\colon\Hom_{\cc}(N,N)\to\Hom_{\cc}(N,L)$ is the same as the map $f_*\colon\Hom_{\ca}(N,N)\to\Hom_{\ca}(N,L)$, which is injective. Therefore the long exact sequence obtained by applying $\Hom_\cc(N,?)$ to the above triangle breaks into two long exact sequences, one of which is
\[
\xymatrix@R=1pc{
\ldots\ar[r] & \Hom_\cc(N,\Sigma^{-1} N)\ar[r] & \Hom_\cc(N,\Sigma^{-1}L)\ar[r] & \Hom_\cc(N,\Sigma^{-1}M)\ar[r] & 0.
}
\]
This implies that
\[
\frac{\{N,L\}}{\{N,N\}}=\frac{\Pi_{l>0}|\Hom_\cc(N,\Sigma^{-l}L)|^{(-1)^l}}{\Pi_{l>0}|\Hom_\cc(N,\Sigma^{-l}N)|^{(-1)^l}}=\Pi_{l>0}|\Hom_\cc(N,\Sigma^{-l}M)|^{(-1)^l}=\{N,M\}.
\]
This proves the second statement and shows that $\{-,-\}$ is multiplicative in the second argument.
Similarly, we show that it is multiplicative in the first argument.
\end{proof}

When $\ca$ is the heart of a $t$-structure, then $\{M,N\}=1$ for $M,N\in\ca$, so $[M][N]=[M]\diamond[N]$. This equality was observed in \cite[Section 6]{Toen06}.

\subsection{Twisted root categories}
\label{ss:twisted-root-category}
Let $\mathcal{A}$ be a Hom-finite and Ext-finite skeletally small hereditary abelian $\fk$-category, and $\sigma$ be a $\fk$-linear automorphism of $\mathcal{A}$. $\sigma$ induces a triangle automorphism of the bounded derived category $\mathcal{D}=\cd^b(\mathcal{A})$ of $\mathcal{A}$. which we also denote by $\sigma$. We have $\Sigma\circ \sigma=\sigma\circ\Sigma$.

Let $d\neq 0$ be an integer, set $F=\Sigma^d\circ \sigma^{-1}$ and consider the orbit category  $\mathcal{C}=\mathcal{D}/F$: the objects in $\mathcal{C}$ are the objects $X$ in $\mathcal{D}$ and morphisms are given by
\[{\rm Hom}_\mathcal{C}(X,Y)=\bigoplus_{p\in \mathbb{Z}}{\rm Hom}_\mathcal{D}(X,F^pY).\]
According to \cite[Section 9]{Keller05} (see also \cite[Section 4]{ChenYang25}), there is a canonical triangle structure on $\mathcal{C}$ such that the canonical projection functor $\cd\to\cc$ is a triangle functor. We call this triangulated category $\cc$  the \emph{$d$-root category twisted by $\sigma^{-1}$}. The automorphism $\sigma$ of $\cd$ induces an automorphism $\sigma$ of $\cc$, still denoted by $\sigma$. The shift functor $\Sigma$ of $\cd$ induces the shift functor $\Sigma$ of $\cc$. For an object $X$ of $\cc$ we have $\Sigma^d(X)\cong \sigma(X)$, however, the functors $\Sigma^d$ and $\sigma$ are in general not isomorphic. When $\sigma=\mathrm{Id}_\ca$, this was introduced in \cite{PengXiao97,PengXiao00} as the root category for $d=2$, and in \cite{Zhang25} as the periodic derived category for general $d$. Because $\mathcal{D}/F=\mathcal{D}/F^{-1}$, we may assume $d>0$.

Assume $d>0$. Below are some homological properties of $\cc$.

\begin{lemma}
\label{lem:when-is-root-category-left-homologically-finite}
\begin{itemize}
\item[(a)] Let $M,N\in\ca$ and $l\in\mathbb{Z}$. If $d=1$, then
\[
\Hom_\cc(M,\Sigma^l N)=\Hom_\ca(M,\sigma^l N)\oplus \Ext^1_\ca(M,\sigma^{l-1}N).
\]
If $d\geq 2$, then
\[
\Hom_\cc(M,\Sigma^{l}N)=\begin{cases} 0 & \text{ if } l\neq pd, pd+1 \text{ for any }p\in\mathbb{Z}\\
\Hom_\ca(M,\sigma^{p}N) & \text{ if } l=pd \text{ for some }p\in\mathbb{Z}\\
\Ext^1_\ca(M,\sigma^{p}N) & \text{ if }l=pd+1 \text{ for some }p\in\mathbb{Z}
\end{cases}.
\]
\item[(b)] $\cc$ is Hom-finite and Krull--Schmidt. Moreover, $\ind(\cc)=\sqcup_{j=0}^{d-1}\Sigma^{j}\ind(\ca)$.
\item[(c)]
$\cc$ is left locally homologically finite if and only if
$\sigma$ satisfies the condition (HF) in Section~\ref{ss:orbit-Euler-form}.
\item[(d)] Let $M,N\in\ca$. Then $\{M,N\}=q^{\langle M,N\rangle_o}$. Further, if $d\geq 2$ and $1\leq t\leq d-1$, then
\begin{align*}
\{M,\Sigma^tN\}&=\begin{cases}\frac{1}{\{M,N\}|\Hom_\ca(M,N)|} & \text{if }t=1\\ q^{(-1)^{d-t}\langle\sigma^{-1}M,N\rangle_o} & \text{if }2\leq t\leq d-1\end{cases},\\
\{\Sigma^t M,N\}&=\begin{cases} \{M,N\}^{(-1)^t} & \text{if }1\leq t\leq d-2\\
\{M,N\}^{(-1)^{d-1}}|\Ext^1_\ca(M,\sigma^{-1}N)|^{-1} & \text{if }t=d-1\end{cases}.
\end{align*}
\end{itemize}
\end{lemma}
\begin{proof}
(a) We have
\begin{align*}
\Hom_\cc(M,\Sigma^{l}N)&=\bigoplus_{p\in\mathbb{Z}}\Hom_\cd(M,F^p(\Sigma^{l}N))=\bigoplus_{p\in\mathbb{Z}}\Hom_\cd(M,\Sigma^{pd+l}\circ\sigma^{-p}N)
\end{align*}
Since $\ca$ is hereditary, it follows that if $d=1$, this is
\begin{align*}
\Hom_\cd(M,\sigma^l N)\oplus \Hom_\cd(M,\Sigma\sigma^{l-1}N)=\Hom_\ca(M,\sigma^l N)\oplus \Ext^1_\ca(M,\sigma^{l-1}N);
\end{align*}
if $d\geq 2$, this is
\begin{align*}
&\hspace{15pt}\begin{cases} 0 & \text{ if } l\neq pd, pd+1 \text{ for any }p\in\mathbb{Z}\\
\Hom_\cd(M,\sigma^{p}N) & \text{ if } l=pd \text{ for some }p\in\mathbb{Z}\\
\Hom_\cd(M,\Sigma \sigma^{p}N) & \text{ if }l=pd+1 \text{ for some }p\in\mathbb{Z}
\end{cases}\\
&=\begin{cases} 0 & \text{ if } l\neq pd, pd+1 \text{ for any }p\in\mathbb{Z}\\
\Hom_\ca(M,\sigma^{p}N) & \text{ if } l=pd \text{ for some }p\in\mathbb{Z}\\
\Ext^1_\ca(M,\sigma^{p}N) & \text{ if }l=pd+1 \text{ for some }p\in\mathbb{Z}
\end{cases}.
\end{align*}

(b) That $\cc$ is Hom-finite follows from (a) and the fact that any object of $\cc$ is isomorphic to $M_0\oplus \Sigma M_1\oplus\ldots\oplus \Sigma^{d-1}M_{d-1}$ for some $M_0,M_1,\ldots,M_{d-1}\in\ca$. For $M\in\ca$ and $0\leq j\leq d-1$, by (a) we have
\[
\End_\cc(\Sigma^j M)\cong\End_\cc(M)=\begin{cases} \End_\ca(M)\ltimes \Ext^1_\ca(M,\sigma^{-1}M) & \text{if }d=1\\
\End_\ca(M) & \text{if }d\geq 2
\end{cases}.
\]
Therefore $\End_\cc(\Sigma^j M)$ is local if and only if $\End_\ca(M)$ is local. It follows that $\cc$ is Krull--Schmidt, and $\ind(\cc)=\sqcup_{j=0}^{d-1}\Sigma^{j}\ind(\ca)$.

(c) Since $\cc$ is Hom-finite, $\cc$ is left locally homologically finite if and only if $\bigoplus_{l\in\mathbb{Z}^+}\Hom_\cc(M,\Sigma^{-l}\Sigma^j N)$ is finite-dimensional for any $M,N\in\ca$ and $0\leq j\leq d-1$, if and only if $\bigoplus_{l\in\mathbb{Z}^+}\Hom_\cc(M,\Sigma^{-l}N)$ is finite-dimensional for any $M,N\in\ca$. By (a),
\begin{align*}
\bigoplus_{l\in\mathbb{Z}^+}\Hom_\cc(M,\Sigma^{-l}N)&=\begin{cases}\bigoplus_{p\geq 1}\Hom_\ca(M,\sigma^{-p}N)\oplus\bigoplus_{p\geq 1}\Ext^1_\ca(M,\sigma^{-p}N) & \text{if }d\geq 2\\
\bigoplus_{p\geq 1}\Hom_\ca(M,\sigma^{-p}N)\oplus\bigoplus_{p\geq 2}\Ext^1_\ca(M,\sigma^{-p}N) & \text{if }d=1.
\end{cases}
\end{align*}
The desired result follows.

(d) We have
\begin{align*}
\{M,N\}&=\Pi_{l>0}|\Hom_\cc(M,\Sigma^{-l}N)|^{(-1)^l}\\
&=\Pi_{p\in\mathbb{Z}^+}|\Hom_\mathcal{A}(M,\sigma^{-p}N)|^{(-1)^{dp}}\cdot|\Ext^1_\mathcal{A}(M,\sigma^{-p}N)|^{(-1)^{dp-1}}\\
&=\Pi_{p\in\mathbb{Z}^+}[\frac{|\Hom_\mathcal{A}(M,\sigma^{-p}N)|}{|\Ext^1_\mathcal{A}(M,\sigma^{-p}N)|}]^{(-1)^{dp}}=\Pi_{p\in\mathbb{Z}^+}[q^{\langle M,\sigma^{-p}N \rangle}]^{(-1)^{dp}}\\
&=q^{\langle M,N\rangle_o},
\end{align*}
where the second equality follows from (a). By Lemma~\ref{lemma-derived-Hall-number-for-direct-sum} and (a), we have
\begin{align*}
\{M,\Sigma^t N\}&=\{M,N\}^{(-1)^t}\cdot(\Pi_{l=-t+1}^0|\Hom_\cc(M,\Sigma^{-l}N)|^{(-1)^l})^{(-1)^t}\\
\end{align*}
Therefore, when $t=1$,
\[
\{M,\Sigma N\}=\{M,N\}^{-1}|\Hom_\cc(M,N)|^{-1}=\frac{1}{\{M,N\}|\Hom_\ca(M,N)|};
\]
when $2\leq t\leq d-1$,
\begin{align*}
\{M,\Sigma^t N\}&=\{M,N\}^{(-1)^t}(\frac{|\Hom_\cc(M,N)|}{|\Hom_\cc(M,\Sigma N)|})^{(-1)^t}\\
&=\{M,N\}^{(-1)^t}(\frac{|\Hom_\ca(M,N)|}{|\Ext^1_\ca(M,N)|})^{(-1)^t}\\
&=(q^{\langle M,N\rangle_o}q^{\langle M,N\rangle})^{(-1)^t}=(q^{\langle M,N\rangle_o+\langle M,N\rangle})^{(-1)^t}\\
&=q^{(-1)^{d-t}\langle\sigma^{-1}M,N\rangle_o},
\end{align*}
where the last equality follows from Lemma~\ref{lem:rotate-orbital-Euler-form}. Also
\begin{align*}
\{\Sigma^t M,N\}&=\{M,N\}^{(-1)^t}\cdot (\Pi_{l=1}^t |\Hom_\cc(M,\Sigma^{-l}N)|)^{(-1)^{l+t-1}}\\
&=\{M,N\}^{(-1)^t}|\Hom_\cc(M,\Sigma^{-d+1}N)|^{-\delta_{t,d-1}},\\
&=\{M,N\}^{(-1)^t}|\Ext^1_\ca(M,\sigma^{-1}N)|^{-\delta_{t,d-1}},
\end{align*}
as desired.
\end{proof}

If $d\geq 2$, then by Lemma~\ref{lem:when-is-root-category-left-homologically-finite} (a), the composition $\ca\to\cd\to\cc$ of the projection functor $\cd\to\cc$ with the embedding $\ca\to\cd$ is fully faithful. In this way, we consider $\ca$ as a subcategory of $\cc$.

\begin{lemma}
\label{lem:A-proper-abelian}
Assume $d\geq 2$.
Then $\ca$ is an extension-closed proper abelian subcategory of $\cc$.
\end{lemma}
\begin{proof}
Let $M,N\in\ca$.
By Lemma~\ref{lem:when-is-root-category-left-homologically-finite} (a), we have
\begin{align*}
\Hom_\mathcal{C}(M,\Sigma N)=\Ext^1_\mathcal{A}(M,N).
\end{align*}
If there is a short exact sequence $\xymatrix@C=0.7pc{0\ar[r] & N\ar[r]^{f} & L\ar[r]^{g} & M\ar[r] & 0}$ in $\ca$, then $L\in\ca$ and this short exact sequence extends to a triangle $\xymatrix@C=0.7pc{N\ar[r]^{f} & L\ar[r]^{g} & M\ar[r] & \Sigma N}$ in $\cd$, and further induces a triangle $\xymatrix@C=0.7pc{N\ar[r]^{f} & L\ar[r]^{g} & M\ar[r] & \Sigma N}$ in $\cc$. Conversely,
for $L\in\cc$, assume that there is a triangle $\xymatrix@C=0.7pc{N\ar[r]^f & L\ar[r]^g & M\ar[r]^h & \Sigma N}$ in $\cc$. Then $h\in\Ext^1_\ca(M,N)$, and thus there is a short exact sequence $\xymatrix@C=0.7pc{0\ar[r] & N\ar[r]^{f'} & L'\ar[r]^{g'} & M\ar[r] & 0}$ in $\ca$ and a triangle $\xymatrix@C=0.7pc{N\ar[r]^{f'} & L'\ar[r]^{g'} & M\ar[r]^h & \Sigma N}$ in $\cd$, which induces a triangle $\xymatrix@C=0.7pc{N\ar[r]^{f'} & L'\ar[r]^{g'} & M\ar[r]^h & \Sigma N}$ in $\cc$. This triangle has to be isomorphic to the given triangle. This shows that $\ca$ is an extension-closed proper abelian subcategory.
\end{proof}

\subsection{The derived Hall algebra of a twisted root category: relations}
\label{ss:Hall-algebra-of-twisted-root-category-relations}

Let $\ca,\cd,\sigma,d,\cc$ be as in Section~\ref{ss:twisted-root-category}.
In this subsection and the next subsection, we study the derived Hall algebra $\ch(\cc)$ of $\cc$. Assume that $\sigma$ satisfies the condition (HF) in Section~\ref{ss:orbit-Euler-form}. In this case, $\cc$ is left locally homologically finite by Lemma~\ref{lem:when-is-root-category-left-homologically-finite} (c), and $\ch(\cc)$ is well-defined. 

Assume $d\geq 2$. We first compute the products of shifts of objects of $\ca$.

\begin{proposition}
\label{prop:Hall-subalgebra-in-root-category}
Let $M,N\in\ca$.
\begin{itemize}
\item[(a)] $[M][N]=q^{\langle N,M\rangle_o}[M]\diamond[N]$.
\item[(b)] $[\Sigma M][N]=\sum_{[C],[K]\in\cp(\ca)}F_{\Sigma M,N}^{C\oplus \Sigma K}[C\oplus \Sigma K]$. Moreover,
\begin{align*}
F_{\Sigma M,N}^{C\oplus \Sigma K}&=|{}_K\Hom_\ca(M,N)_C|\cdot\frac{|\Ext^1_\ca(C,K)|}{|\Hom_\ca(C,K)|}\cdot\frac{|\Aut_\ca(C)|\cdot |\Aut_\ca(K)|}{|\Aut_\ca(M)|\cdot |\Aut_\ca(N)|}\\
&\hspace{15pt}\cdot\frac{\{M,N\}}{\{C,K\}\{K,C\}}\cdot \frac{\{C,C\}\{K,K\}}{\{M,M\}\{N,N\}},
\end{align*}
where ${}_K\Hom_\ca(M,N)_C$ is the set of morphisms from $M$ to $N$ whose kernel is isomorphic to $K$ and whose cokernel is isomorphic to $C$. In particular,
\begin{align*}
F_{\Sigma M,N}^{N\oplus\Sigma M}&=q^{(-1)^{d-1}\langle \sigma^{-1}N,M\rangle_o},\\
F_{\Sigma M,M}^0&=\frac{1}{|\Aut_\ca(M)|q^{\langle M,M\rangle_o}}.
\end{align*}
\item[(c)] For $2\leq t\leq d-1$, we have
\[
[\Sigma^t M][N]=q^{(-1)^{{d-t}}\langle \sigma^{-1}N,M\rangle_o}[N\oplus \Sigma^t M].
\]
\end{itemize}
\end{proposition}
\begin{proof}
(a) follows immediately from Lemma~\ref{lem:derived-Hall-number-for-proper-abelian-subcategory} and Lemma~\ref{lem:when-is-root-category-left-homologically-finite} (d).

(b)
First, $\Hom_\cc(\Sigma M,\Sigma N)\cong\Hom_\cc(M,N)=\Hom_\ca(M,N)$ by Lemma~\ref{lem:when-is-root-category-left-homologically-finite}. For a morphism $f\in\Hom_\ca(M,N)$, it induces a triangle in $\cd$ and hence in $\cc$
\[
\xymatrix{
M\ar[r]^f & N\ar[r] & \cok(f)\oplus \Sigma\ker(f)\ar[r] & \Sigma M,
}
\]
showing that the cone of $f$ in $\cc$ is $X\cong \cok(f)\oplus\Sigma\ker(f)$. This implies the equality on $[\Sigma M][N]$. For the equality on $F_{\Sigma M,N}^{C\oplus \Sigma K}$:
\begin{align*}
F_{\Sigma M,N}^{C\oplus \Sigma K}&=|\Hom_\cc(M,N)_{C\oplus\Sigma K}|\cdot\{M,N\}\cdot \frac{|\Aut_\cc(C\oplus\Sigma K)|}{|\Aut_\cc(M)|\cdot|\Aut_\cc(N)|}\cdot\frac{\{C\oplus\Sigma K,C\oplus\Sigma K\}}{\{M,M\}\{N,N\}}\\
&=|{}_K\Hom_\ca(M,N)_C|\cdot\{M,N\}\cdot \frac{|\Aut_\cc(C\oplus\Sigma K)|}{|\Aut_\cc(M)|\cdot|\Aut_\cc(N)|}\cdot\frac{\{C\oplus\Sigma K,C\oplus\Sigma K\}}{\{M,M\}\{N,N\}}\\
&=|{}_K\Hom_\ca(M,N)_C|\cdot\{M,N\}\cdot \frac{|\Aut_\cc(C\oplus\Sigma K)|}{|\Aut_\ca(M)|\cdot|\Aut_\ca(N)|}\\
&\hspace{15pt}\cdot\frac{\{C,C\}\{C,\Sigma K\}\{\Sigma K,C\}\{\Sigma K,\Sigma K\}}{\{M,M\}\{N,N\}}\\
&=|{}_K\Hom_\ca(M,N)_C|\cdot\{M,N\}\cdot \frac{|\Aut_\cc(C\oplus\Sigma K)|}{|\Aut_\ca(M)|\cdot|\Aut_\ca(N)|}\\
&\hspace{15pt}\cdot\frac{\{C,C\}\{K,K\}}{\{M,M\}\{C,K\}|\Hom_\cc(C,K)|\{K,C\}|\Hom_\cc(K,\Sigma^{-1} C)|\{N,N\}},
\end{align*}
where first equality is the derived Riedtmann--Peng formula, the second equality holds because of the equality $\Hom_\cc(M,N)_{C\oplus \Sigma K}={}_K\Hom_\ca(M,N)_C$,
the third and the fourth equalities follow by Lemma~\ref{lemma-derived-Hall-number-for-direct-sum}.
Moreover, since indecomposable direct summands of $\Sigma K$ and indecomposable direct summands of $C$ are not isomorphic, we have
\begin{align*}
\Aut(C\oplus\Sigma K)&=\begin{pmatrix} \Aut_\cc(C) & \Hom_\cc(\Sigma K,C)\\ \Hom_\cc(C,\Sigma K) & \Aut_\cc(\Sigma K) \end{pmatrix},\\
|\Aut_\cc(C\oplus\Sigma K)|&=|\Aut_\cc(C)|\cdot|\Hom_\cc(\Sigma K,C)|\cdot|\Hom_\cc(C,\Sigma K)|\cdot|\Aut_\cc(K)|\\
&=|\Aut_\ca(C)|\cdot|\Hom_\cc(\Sigma K,C)|\cdot|\Hom_\cc(C,\Sigma K)|\cdot|\Aut_\ca(K)|.
\end{align*}
It follows that
\begin{align*}
F_{\Sigma M,N}^{C\oplus \Sigma K}&=|{}_K\Hom_\ca(M,N)_C|\cdot\frac{|\Hom_\cc(C,\Sigma K)|}{|\Hom_\cc(C,K)|}\cdot\frac{|\Aut_\ca(C)|\cdot |\Aut_\ca(K)|}{|\Aut_\ca(M)|\cdot |\Aut_\ca(N)|}\\
&\hspace{15pt}\cdot\frac{\{M,N\}}{\{C,K\}\{K,C\}}\cdot \frac{\{C,C\}\{K,K\}}{\{M,M\}\{N,N\}}\\
&=|{}_K\Hom_\ca(M,N)_C|\cdot\frac{|\Ext^1_\ca(C,K)|}{|\Hom_\ca(C,K)|}\cdot\frac{|\Aut_\ca(C)|\cdot |\Aut_\ca(K)|}{|\Aut_\ca(M)|\cdot |\Aut_\ca(N)|}\\
&\hspace{15pt}\cdot\frac{\{M,N\}}{\{C,K\}\{K,C\}}\cdot \frac{\{C,C\}\{K,K\}}{\{M,M\}\{N,N\}}.
\end{align*}

For the equality on $F_{\Sigma M,N}^{N\oplus\Sigma M}$: putting $C=N$ and $K=M$ in the formula for $F_{\Sigma M,N}^{C\oplus \Sigma K}$, we obtain
\begin{align*}
F_{\Sigma M,N}^{N\oplus \Sigma M}&=|{}_M\Hom_\ca(M,N)_N|\cdot\frac{|\Ext^1_\ca(N,M)|}{|\Hom_\ca(N,M)|}\cdot\frac{|\Aut_\ca(N)|\cdot |\Aut_\ca(M)|}{|\Aut_\ca(M)|\cdot |\Aut_\ca(N)|}\\
&\hspace{15pt}\cdot\frac{\{M,N\}}{\{N,M\}\{M,N\}}\cdot \frac{\{N,N\}\{M,M\}}{\{M,M\}\{N,N\}},\\
&=\frac{|\Ext^1_\ca(N,M)|}{\{N,M\}\cdot |\Hom_\ca(N,M)|}=\frac{1}{q^{\langle N,M\rangle_o}q^{\langle N,M\rangle}}\\
&=\frac{1}{q^{(-1)^d \langle \sigma^{-1}N,M\rangle_o}}=q^{(-1)^{d-1} \langle \sigma^{-1}N,M\rangle_o}.
\end{align*}
Here the second equality follows from the fact ${}_M\Hom_\ca(M,N)_N=0$, the third equality follows from Lemma~\ref{lem:when-is-root-category-left-homologically-finite} (d), and the fourth equality follows from Lemma~\ref{lem:rotate-orbital-Euler-form}.
For the equality on $F_{\Sigma M,M}^0$: putting $N=M$ and $C=K=0$ in the formula for $F_{\Sigma M,N}^{C\oplus\Sigma K}$ and replacing the terms in which $0$ appears by $1$, we obtain
\begin{align*}
F_{\Sigma M,M}^0&=|{}_0\Hom_\ca(M,M)_0|\cdot\frac{1}{1}\cdot\frac{1}{|\Aut_\ca(M)|\cdot|\Aut_\ca(M)|}\cdot\frac{\{M,M\}}{1}\cdot \frac{1}{\{M,M\}\cdot\{M,M\}}\\
&=\frac{1}{|\Aut_\ca(M)|\{M,M\}}=\frac{1}{|\Aut_\ca(M)|q^{\langle M,M\rangle_o}}.
\end{align*}
Here the second equality follows from the fact ${}_0\Hom_\ca(M,M)_0=\Aut_\ca(M)$.

(c) Since $\Hom_\cc(\Sigma^t M,\Sigma N)\cong\Hom_\cc(M,\Sigma^{-t+1}N)=0$ by Lemma~\ref{lem:when-is-root-category-left-homologically-finite} (a), we have
\[
[\Sigma^t M][N]=F_{\Sigma^t M,N}^{N\oplus\Sigma^t M}[N\oplus \Sigma^t M].
\]
Now
\begin{align*}
F_{\Sigma^t M,N}^{N\oplus\Sigma^t M}&=\frac{|\Hom_\cc(N,N\oplus \Sigma^t M)_{\Sigma^t M}|}{|\Aut_\cc(N)|}\cdot\frac{\{N,N\oplus\Sigma^t M\}}{\{N,N\}}\\
&=\frac{|\Aut_\cc(N)|}{|\Aut_\cc(N)|}\cdot\{N,\Sigma^t M\}=q^{(-1)^{d-t}\langle\sigma^{-1}N,M\rangle_o}.
\end{align*}
Here the second equality follows from Lemma~\ref{lemma-derived-Hall-number-for-direct-sum} and the fact $\Hom_\cc(N,\Sigma^t M)=0$, and the last equality follows from Lemma~\ref{lem:when-is-root-category-left-homologically-finite} (d).
\end{proof}

We deduce the following useful corollaries.

\begin{corollary}
\label{cor:commutator-for-neighbouring-pieces-simples}
Let $S,S'$ be simple objects of $\ca$. Then
\[
[\Sigma S][S']=\begin{cases} q^{(-1)^{d-1}\langle\sigma^{-1}S',S\rangle_o}[S'\oplus\Sigma S]+\frac{1}{|\Aut(S)|q^{\langle S,S\rangle_o}}[0] & \text{if }S'\cong S\\
q^{(-1)^{d-1}\langle\sigma^{-1}S',S\rangle_o}[S'\oplus\Sigma S] & \text{if }S'\not\cong S
\end{cases}.
\]
If $d=2$, then
\[
[\Sigma S][S']-q^{(S,S')_o-\langle S',S\rangle}[S'][\Sigma S]=
\begin{cases} \frac{1}{|\Aut(S)|q^{\langle S,S\rangle_o}}[0] & \text{if }S'\cong S\\[5pt]
-\frac{q^{\langle S,\sigma S\rangle_o-2\langle S,S\rangle_o}}{|\Aut(S)|}[0] & \text{if }S'\cong\sigma S\\
0 &\text{otherwise}
\end{cases};
\]
if $d>2$, then
\[
[\Sigma S][S']-q^{(S,S')_o-\langle S',S\rangle}[S'][\Sigma S]=
\begin{cases} \frac{1}{|\Aut(S)|q^{\langle S,S\rangle_o}}[0] & \text{if }S'\cong S\\
0 &\text{otherwise}
\end{cases}.
\]
\end{corollary}
\begin{proof}
The first equality follows from Proposition~\ref{prop:Hall-subalgebra-in-root-category} (b), observing that ${}_K\Hom_\ca(S,S')_C$ is empty except in the following three cases:
\begin{itemize}
\item[-] $S\cong S'$, $K=S$ and $C=S$;
\item[-] $S\cong S'$, $K=0$ and $C=0$;
\item[-] $S\not\cong S'$, $K=S$ and $C=S'$.
\end{itemize}

Next assume $d=2$. By applying the automorphism $\Sigma$ of $\ch(\cc)$ to
\begin{align*}
[\Sigma \sigma^{-1}S'][S]&=\begin{cases} q^{-\langle\sigma^{-1}S,\sigma^{-1} S'\rangle_o}[S\oplus\Sigma \sigma^{-1} S']+\frac{1}{|\Aut(S)|q^{\langle S,S\rangle_o}}[0] & \text{if }S'\cong \sigma S\\
q^{-\langle\sigma^{-1}S,\sigma^{-1}S'\rangle_o}[S\oplus\Sigma \sigma^{-1}S'] & \text{if }S'\not\cong \sigma S
\end{cases}\\
&=\begin{cases} q^{-\langle S,S'\rangle_o}[S\oplus\Sigma \sigma^{-1}S']+\frac{1}{|\Aut(S)|q^{\langle S,S\rangle_o}}[0] & \text{if }S'\cong \sigma S\\
q^{-\langle S,S'\rangle_o}[S\oplus\Sigma \sigma^{-1}S'] & \text{if }S'\not\cong \sigma S
\end{cases},
\end{align*}
we obtain
\begin{align*}
[S'][\Sigma S]&=[\Sigma^2 \sigma^{-1}S'][\Sigma S]=
\begin{cases} q^{-\langle S,S'\rangle_o}[S'\oplus\Sigma S]+\frac{1}{|\Aut(S)|q^{\langle S,S\rangle_o}}[0] & \text{if }S'\cong \sigma S\\
q^{-\langle S,S'\rangle_o}[S'\oplus\Sigma S] & \text{if }S'\not\cong \sigma S
\end{cases}.
\end{align*}
Therefore
\[
[\Sigma S][S']-q^{-\langle \sigma^{-1}S',S\rangle_o+\langle S,S'\rangle_o}[S'][\Sigma S]=
\begin{cases} \frac{1}{|\Aut(S)|q^{\langle S,S\rangle_o}}[0] & \text{if }S'\cong S\\[5pt]
-\frac{q^{\langle S,\sigma S\rangle_o-2\langle S,S\rangle_o}}{|\Aut(S)|}[0]& \text{if }S'\cong\sigma S\\
0 &\text{otherwise}
\end{cases}.
\]
The second equality then follows from Lemma~\ref{lem:rotate-orbital-Euler-form}.

Finally assume $d>2$. By Proposition~\ref{prop:Hall-subalgebra-in-root-category} (c), we have
\[
[\Sigma^{d-1}\sigma^{-1}S'][S]=q^{-\langle \sigma^{-1}S,\sigma^{-1}S'\rangle_o}[S\oplus\Sigma^{d-1} \sigma^{-1}S']=q^{-\langle S,S'\rangle_o}[S\oplus\Sigma^{d-1} \sigma^{-1}S'].
\]
By applying the automorphism $\Sigma$ of $\ch(\cc)$ to it, we obtain
\[
[S'][\Sigma S]=[\Sigma^{d}\sigma^{-1}S'][\Sigma S]=q^{-\langle S,S'\rangle_o}[S'\oplus\Sigma S].
\]
Therefore,
\[
[\Sigma S][S']-q^{(-1)^{d-1}\langle\sigma^{-1}S',S\rangle_o+\langle S,S'\rangle_o}[S'][\Sigma S]=
\begin{cases} \frac{1}{|\Aut(S)|q^{\langle S,S\rangle_o}}[0] & \text{if }S'\cong S\\
0 &\text{otherwise}
\end{cases}.
\]
The third equality then follows from Lemma~\ref{lem:rotate-orbital-Euler-form}.
\end{proof}

\begin{corollary}
\label{cor:commutator-for-neighbouring-pieces-projectives}
Let $P,P'$ be indecomposable projective objects of $\ca$. Then
\[
[\Sigma P][P']=\sum_{[C]\in\cp(\ca)}F_{\Sigma P,P'}^C[C]+q^{(-1)^{d-1}\langle\sigma^{-1} P',P\rangle_o}[P'\oplus\Sigma P],
\]
where
\[
F_{\Sigma P,P'}^C=|{}_0\Hom_\ca(P,P')_C|\cdot\frac{|\Aut_\ca(C)|}{|\Aut_\ca(P)|\cdot |\Aut_\ca(P')|}\cdot\frac{1}{\{P',P\}}.
\]
If $d=2$, then
\begin{align*}
[\Sigma P][P']-q^{(P,P')_o-\langle P',P\rangle}&[P'][\Sigma P]=
\sum_{[C]\in\cp(\ca)}F_{\Sigma P,P'}^C[C]\\
&\hspace{15pt}-q^{(P,P')_o-\langle P',P\rangle}\sum_{[C]\in\cp(\ca)}F_{\Sigma \sigma^{-1}P',P}^C[\Sigma C];
\end{align*}
if $d>2$, then
\begin{align*}
[\Sigma P][P']-q^{(P,P')_o-\langle P',P\rangle}[P'][\Sigma P]&=\sum_{[C]\in\cp(\ca)}F_{\Sigma P,P'}^C[C].
\end{align*}
\end{corollary}
\begin{proof}
Since $\ca$ is hereditary, it follows that any morphism $P\to P'$ is either a monomorphism or zero. Therefore by Proposition~\ref{prop:Hall-subalgebra-in-root-category} (b), we have
\[
[\Sigma P][P']=\sum_{[C]\in\cp(\ca)}F_{\Sigma P,P'}^{C}[C]+F_{\Sigma P,P'}^{P'\oplus\Sigma P}[P'\oplus \Sigma P].
\]
where  $F_{\Sigma P,P'}^{P'\oplus\Sigma P}=q^{(-1)^{d-1}\langle\sigma^{-1}P',P\rangle_o}$. Moreover,
\begin{align*}
F_{\Sigma P,P'}^{C}&=|{}_0\Hom_\ca(P,P')_C|\cdot\frac{|\Aut_\ca(C)|}{|\Aut_\ca(P)|\cdot |\Aut_\ca(P')|}\cdot
\frac{\{P,P'\}\{C,C\}}{\{P,P\}\{P',P'\}}.
\end{align*}
In this case there is a short exact sequence $\xymatrix@C=0.7pc{0\ar[r] &P\ar[r] & P'\ar[r] & C\ar[r] &0}$ in $\ca$, so by Lemma~\ref{lem:derived-Hall-number-for-proper-abelian-subcategory} we have
\begin{align*}
\frac{\{P,P'\}\{C,C\}}{\{P,P\}\{P',P'\}}&=\frac{\{C,C\}}{\{P,P\}\{C,P'\}}=\frac{1}{\{P,P\}\{C,P\}}=\frac{1}{\{P',P\}}.
\end{align*}
Therefore,
\begin{align*}
F_{\Sigma P,P'}^{C}&=|{}_0\Hom_\ca(P,P')_C|\cdot\frac{|\Aut_\ca(C)|}{|\Aut_\ca(P)|\cdot |\Aut_\ca(P')|}\cdot
\frac{1}{\{P',P\}}.
\end{align*}

Next assume $d=2$. By applying $\Sigma$ to
\begin{align*}
[\Sigma \sigma^{-1}P'][P]&=\sum_{[C]\in\cp(\ca)}F_{\Sigma \sigma^{-1}P',P}^C[C]+q^{-\langle\sigma^{-1} P,\sigma^{-1}P'\rangle_o}[P\oplus\Sigma \sigma^{-1}P']\\
&=\sum_{[C]\in\cp(\ca)}F_{\Sigma \sigma^{-1}P',P}^C[C]+q^{-\langle P,P'\rangle_o}[P\oplus\Sigma \sigma^{-1}P'],
\end{align*}
we obtain
\begin{align*}
[P'][\Sigma P]&=[\Sigma^2 \sigma^{-1}P'][\Sigma P]=\sum_{[C]\in\cp(\ca)}F_{\Sigma \sigma^{-1} P',P}^C[\Sigma C]+q^{-\langle P,P'\rangle_o}[P'\oplus\Sigma P].
\end{align*}
Therefore
\begin{align*}
[\Sigma P][P']-&q^{-\langle\sigma^{-1}P',P\rangle_o+\langle P,P'\rangle_o}[P'][\Sigma P]=
\sum_{[C]\in\cp(\ca)}F_{\Sigma P,P'}^C[C]\\
&\hspace{15pt}-q^{-\langle\sigma^{-1}P',P\rangle_o+\langle P,P'\rangle_o}\sum_{[C]\in\cp(\ca)}F_{\Sigma \sigma^{-1}P',P}^C[\Sigma C].
\end{align*}
The desired formula follows from Lemma~\ref{lem:rotate-orbital-Euler-form}.

Finally assume $d>2$. By Proposition~\ref{prop:Hall-subalgebra-in-root-category} (c), we have
\[
[\Sigma^{d-1}\sigma^{-1}P'][P]=q^{-\langle \sigma^{-1}P,\sigma^{-1}P'\rangle_o}[P\oplus\Sigma^{d-1} \sigma^{-1}P']=q^{-\langle P,P'\rangle_o}[P\oplus\Sigma^{d-1} \sigma^{-1}P'].
\]
By applying the automorphism $\Sigma$ of $\ch(\cc)$ to it, we obtain
\[
[P'][\Sigma P]=[\Sigma^{d}\sigma^{-1}P'][\Sigma P]=q^{-\langle P,P'\rangle_o}[P'\oplus\Sigma P].
\]
Therefore
\[
[\Sigma P][P']-q^{(-1)^{d-1}\langle\sigma^{-1}P',P\rangle_o+\langle P,P'\rangle_o}[P'][\Sigma P]=
\sum_{[C]\in\cp(\ca)}F_{\Sigma P,P'}^C[C].
\]
The desired formula follows from Lemma~\ref{lem:rotate-orbital-Euler-form}.
\end{proof}

\begin{corollary}
\label{cor:commutator-for-neighbouring-pieces-projective-simple}
Let $P\in\ca$ be indecomposable projective and $S\in\ca$ be simple.
\begin{itemize}
\item[(a)]  $
[\Sigma P][S]=\begin{cases}
F_{\Sigma P,S}^{\Sigma\rad(P)}[\Sigma\rad(P)]+q^{(-1)^{d-1}\langle \sigma^{-1}S,P\rangle_o}[S\oplus\Sigma P] & \text{if }S\cong\top(P)\\
q^{(-1)^{d-1}\langle \sigma^{-1}S,P\rangle_o}[S\oplus\Sigma P] & \text{if }S\not\cong \top(P)
\end{cases},
$
where
\[
F_{\Sigma P,S}^{\Sigma\rad(P)}=\frac{|\Aut_\ca(\rad(P))|}{|\Aut_\ca(P)|}\cdot\frac{1}{\{S,P\}}=\frac{|\Aut_\ca(\rad(P))|}{|\Aut_\ca(P)|}q^{-\langle S,P\rangle_o}.
\]
\item[(b)] Assume that $\Hom_\ca(S,P)=0$. Then
$
[\Sigma S][P]=q^{(-1)^{d-1}\langle \sigma^{-1}P,S\rangle_o}[P\oplus\Sigma S].
$
If $d=2$, then
\begin{align*}
[\Sigma S][P]-&q^{(S,P)_o-\langle P,S\rangle}[P][\Sigma S]\\
&=
\begin{cases} -q^{(S,P)_o-\langle P,S\rangle}\frac{|\Aut_\ca(\rad(P))|}{|\Aut_\ca(P)|}q^{-\langle \sigma S,P\rangle_o}[\rad(P)] &\text{if }\sigma S\cong \top(P)\\
0 & \text{if }\sigma S\not\cong\top(P)
\end{cases};
\end{align*}
if $d>2$, then
\[
[\Sigma S][P]-q^{(S,P)_o-\langle P,S\rangle}[P][\Sigma S]=0.
\]

\item[(c)]  If $d>2$, then
\begin{align*}
[\Sigma P][S]-&q^{(P,S)_o-\langle S,P\rangle}[S][\Sigma P]\\
&=
\begin{cases} \frac{|\Aut_\ca(\rad(P))|}{|\Aut_\ca(P)|}q^{-\langle S,P\rangle_o}[\Sigma \rad(P)] & \text{if }S\cong \top(P)\\
0 &\text{if }S\not\cong \top(P)
\end{cases}.
\end{align*}
This is also true for $d=2$, provided that $\Hom_\ca(\sigma^{-1}S,P)=0$.
\end{itemize}
\end{corollary}
\begin{proof}
(a)
If $S\not\cong \top(P)$, then $\Hom_\ca(P,S)=0$ and ${}_K\Hom_\ca(P,S)_C$ is non-empty if and only if $K=P$ and $C=S$.
It follows from Proposition~\ref{prop:Hall-subalgebra-in-root-category} (b) that
\[
[\Sigma P][S]=F_{\Sigma P,S}^{S\oplus\Sigma P}[S\oplus \Sigma P]=q^{(-1)^{d-1}\langle\sigma^{-1}S,P\rangle_o}[S\oplus\Sigma P].
\]
If $S\cong \top(P)$, then $\Hom_\ca(P,S)\cong\End_\ca(S)$, and ${}_K\Hom_\ca(P,S)_C$ is non-empty if and only if $K=P$ and $C=S$, or $K=\rad(P)$ and $C=0$. It follows from Proposition~\ref{prop:Hall-subalgebra-in-root-category} (b) that
\begin{align*}
[\Sigma P][S]&=F_{\Sigma P,S}^{\Sigma \rad(P)}[\Sigma \rad(P)]+F_{\Sigma P,S}^{S\oplus\Sigma P}[S\oplus \Sigma P]\\
&=F_{\Sigma P,S}^{\Sigma \rad(P)}[\Sigma \rad(P)]+q^{(-1)^{d-1}\langle\sigma^{-1}S,P\rangle_o}[S\oplus\Sigma P],
\end{align*}
where
\begin{align*}
F_{\Sigma P,S}^{\Sigma \rad(P)}&=|{}_{\rad(P)}\Hom_\ca(P,S)_0|\cdot\frac{|\Aut_\ca(\rad(P))|}{|\Aut_\ca(P)|\cdot |\Aut_\ca(S)|}\cdot \frac{\{P,S\}\{\rad(P),\rad(P)\}}{\{P,P\}\{S,S\}}\\
&=\frac{|\Aut_\ca(\rad(P))|}{|\Aut_\ca(P)|}\cdot \frac{\{P,S\}\{\rad(P),\rad(P)\}}{\{P,P\}\{S,S\}},
\end{align*}
since ${}_{\rad(P)}\Hom_\ca(P,S)_0=\Aut_\ca(S)$.
Because there is a short exact sequence\\ $\xymatrix@C=0.7pc{0\ar[r] &\rad(P)\ar[r] & P\ar[r] & S\ar[r] &0}$ in $\ca$, by Lemma~\ref{lem:derived-Hall-number-for-proper-abelian-subcategory} we have
\begin{align*}
\frac{\{P,S\}\{\rad(P),\rad(P)\}}{\{P,P\}\{S,S\}}&=\frac{\{\rad(P),\rad(P)\}}{\{P,\rad(P)\}\{S,S\}}=\frac{1}{\{S,\rad(P)\}\{S,S\}}=\frac{1}{\{S,P\}}.
\end{align*}
So
\begin{align*}
F_{\Sigma P,S}^{\Sigma\rad(P)}&=\frac{|\Aut_\ca(\rad(P))|}{|\Aut_\ca(P)|}\cdot \frac{1}{\{S,P\}}.
\end{align*}
(b) The first statement follows directly from Proposition~\ref{prop:Hall-subalgebra-in-root-category} (b).
Next assume $d=2$. By (a) we have
\begin{align*}
[\Sigma \sigma^{-1}P][S]&=\begin{cases}
F_{\Sigma \sigma^{-1}P,S}^{\Sigma\rad(\sigma^{-1}P)}[\Sigma\rad(\sigma^{-1}P)]\\ \hspace{15pt}+q^{(-1)^{d-1}\langle \sigma^{-1}S,\sigma^{-1}P\rangle_o}[S\oplus\Sigma \sigma^{-1}P] & \text{if }S\cong\top(\sigma^{-1}P)\\
q^{(-1)^{d-1}\langle \sigma^{-1}S,\sigma^{-1}P\rangle_o}[S\oplus\Sigma \sigma^{-1}P] & \text{if }S\not\cong \top(\sigma^{-1}P)
\end{cases}\\
&=\begin{cases}
F_{\Sigma P,\sigma S}^{\Sigma\rad(P)}[\Sigma \sigma^{-1}\rad(P)]\\ \hspace{15pt}+q^{-\langle S,P\rangle_o}[S\oplus\Sigma \sigma^{-1}P] & \text{if }\sigma S\cong\top(P)\\
q^{-\langle S,P\rangle_o}[S\oplus\Sigma \sigma^{-1}P] & \text{if }\sigma S\not\cong \top(P)
\end{cases}.
\end{align*}
Applying the automorphism $\Sigma$ to this equality we obtain
\begin{align*}
[P][\Sigma S]&=[\Sigma^2\sigma^{-1}P][\Sigma S]=\begin{cases}
F_{\Sigma P,\sigma S}^{\Sigma\rad(P)}[\rad(P)]\\ \hspace{15pt}+q^{-\langle S,P\rangle_o}[P\oplus\Sigma S] & \text{if }\sigma S\cong\top(P)\\
q^{-\langle S,P\rangle_o}[P\oplus\Sigma S] & \text{if }\sigma S\not\cong \top(P)
\end{cases}.
\end{align*}
Therefore
\begin{align*}
[\Sigma S][P]-&q^{(-1)^{d-1}\langle \sigma^{-1}P,S\rangle_o+\langle S,P\rangle_o}[P][\Sigma S]\\
&=
\begin{cases} -q^{(-1)^{d-1}\langle \sigma^{-1}P,S\rangle_o+\langle S,P\rangle_o}F_{\Sigma P,\sigma S}^{\Sigma\rad(P)}[\rad(P)] &\text{if }\sigma S\cong \top(P)\\
0 & \text{if }\sigma S\not\cong\top(P)
\end{cases}.
\end{align*}
The desired formula follows from (a) and Lemma~\ref{lem:rotate-orbital-Euler-form}.

Finally assume $d>2$. By Proposition~\ref{prop:Hall-subalgebra-in-root-category} (c), we have
\[
[\Sigma^{d-1}\sigma^{-1}P][S]=q^{-\langle\sigma^{-1}S,\sigma^{-1}P\rangle_o}[S\oplus\Sigma^{d-1}\sigma^{-1}P]=q^{-\langle S,P\rangle_o}[S\oplus\Sigma^{d-1}\sigma^{-1}P].
\]
Applying the automorphism $\Sigma$ to this equality we obtain
\begin{align*}
[P][\Sigma S]&=[\Sigma^d\sigma^{-1}P][\Sigma S]=q^{-\langle S,P\rangle_o}[P\oplus\Sigma S].
\end{align*}
Therefore
\begin{align*}
[\Sigma S][P]-q^{(-1)^{d-1}\langle \sigma^{-1}P,S\rangle_o+\langle S,P\rangle_o}[P][\Sigma S]=0
\end{align*}
The desired formula follows from Lemma~\ref{lem:rotate-orbital-Euler-form}.

(c) If $d>2$, by Proposition~\ref{prop:Hall-subalgebra-in-root-category} (c) we have

\[[\Sigma^{d-1} \sigma^{-1}S][P]=q^{-\langle \sigma^{-1}P,\sigma^{-1}S\rangle_o}[P\oplus\Sigma^{d-1} \sigma^{-1}S]=q^{-\langle P,S\rangle_o}[P\oplus\Sigma^{d-1} \sigma^{-1}S].\]

This is also true for $d=2$ by (b) under the extra assumption.  Applying the automorphism $\Sigma$ to this equality we obtain
\[
[S][\Sigma P]=q^{-\langle P,S\rangle_o}[S\oplus\Sigma P].
\]
In conjunction with (a), this implies
\begin{align*}
[\Sigma P][S]-&q^{(-1)^{d-1}\langle \sigma^{-1}S,P\rangle_o+\langle P,S\rangle_o}[S][\Sigma P]\\
&=
\begin{cases} F_{\Sigma P,S}^{\Sigma\rad(P)}[\Sigma \rad(P)] & \text{if }S\cong \top(P)\\
0 &\text{if }S\not\cong \top(P)
\end{cases}.
\end{align*}
The desired formula then follows from Lemma~\ref{lem:rotate-orbital-Euler-form} and the description of $F_{\Sigma P,S}^{\Sigma\rad(P)}$ in (a).
\end{proof}

\begin{corollary}
\label{cor:commutator-for-pieces-faraway}
Let $M,N\in\ca$. For $2\leq t\leq d-2$, we have
\[
[\Sigma^t M][N]-q^{(-1)^{t-1}((M,N)_o-\langle N,M\rangle)}[N][\Sigma^t M]=0.
\]
\end{corollary}
\begin{proof}
By Proposition~\ref{prop:Hall-subalgebra-in-root-category} (c), we have
\begin{align*}
[\Sigma^t M][N]&=q^{(-1)^{d-t}\langle \sigma^{-1}N,M\rangle_o}[N\oplus\Sigma^t M],\\
[\Sigma^{d-t}\sigma^{-1}N][M]&=q^{(-1)^{t}\langle \sigma^{-1}M,\sigma^{-1}N\rangle_o}[M\oplus\Sigma^{d-t} \sigma^{-1}N]\\
&=q^{(-1)^{t}\langle M,N\rangle_o}[M\oplus\Sigma^{d-t} \sigma^{-1}N]
\end{align*}
Applying the automorphism $\Sigma^t$ of $\ch(\cc)$ to the second equality, we obtain
\[
[N][\Sigma^t M]=[\Sigma^d\sigma^{-1}N][\Sigma^t M]=q^{(-1)^{t}\langle M,N\rangle_o}[N\oplus\Sigma^t M].
\]
Therefore
\[
[\Sigma^t M][N]-q^{(-1)^{d-t}\langle \sigma^{-1}N,M\rangle_o-(-1)^{t}\langle M,N\rangle_o}[N][\Sigma^t M]=0.
\]
We obtain the desired equality by Lemma~\ref{lem:rotate-orbital-Euler-form}.
\end{proof}

\subsection{The derived Hall algebra of a twisted root category: basis}
\label{ss:Hall-algebra-of-twisted-root-category-basis}
Keep the notation and assumptions as in the first paragraph of Section~\ref{ss:Hall-algebra-of-twisted-root-category-relations}.
In this subsection we present a basis for $\ch(\cc)$. It is clear that $\cp(\ca)$ is a basis of $\ch(\cc)$ when $d=1$. So we assume $d\geq 2$.

\begin{lemma}
\label{lemma-Hall-alg-basis-1}
The set $\{[M_0][\Sigma M_1] \cdots [\Sigma^{d-1} M_{d-1}]\mid M_0,M_1,\ldots,M_{d-1}\in\ca\}$ is linearly independent in $\ch(\cc)$.
\end{lemma}
\begin{proof} Following \cite{JensenSuZimmermann05}, we define a partial order $\leq_{\Delta}$ on $\mathcal{P(C)}$ by setting $[Y]\leq_{\Delta}[X]$ if there exists an object $Z$ of $\mathcal{C}$ and a triangle in $\mathcal{C}$:
\[X\rightarrow Y\oplus Z \rightarrow Z\rightarrow \Sigma X. \]
Note that if there is a triangle $\xymatrix@R=0.7pc{X\ar[r]^f & Y\ar[r]^g & Z\ar[r]^h & \Sigma X}$, then $Y<_{\Delta} X\oplus Z$ unless $Y\cong X\oplus Z$, due to the triangle
\[
\xymatrix{
X\oplus Z\ar[r]^{{\scriptsize\begin{pmatrix}f & 0\\ 0 & \mathrm{id}\end{pmatrix}}} & Y\oplus Z \ar[r]^(0.6){(g,0)} & Z\ar[r]^(0.3){{h\choose 0}} & \Sigma(X\oplus Z).
}
\]
We extend this partial order to a total order $\prec$.

Now let $M_{i,j}$ ($1\leq i\leq r$, $0\leq j\leq d-1$) be objects of $\mathcal{A}$. Suppose that $\lambda_1,\ldots, \lambda_r$ are rational numbers such that
\[\lambda_1 [M_{1,0}][\Sigma M_{1,1}]\cdots[\Sigma^{d-1} M_{1,d-1}]+\ldots +\lambda_r [M_{r,0}][\Sigma M_{r,1}]\cdots[\Sigma^{d-1} M_{r,d-1}]=0.\]
There is a unique maximal element among all the $[M_{i,0}\oplus\Sigma M_{i,1}\oplus \cdots \oplus \Sigma^{d-1}M_{i,d-1}]$'s, say $[M_{1,0}\oplus\Sigma M_{1,1}\oplus \cdots \oplus \Sigma^{d-1}M_{1,d-1}]$. Then we have
\begin{align*}
&\lambda_1 [M_{1,0}][\Sigma M_{1,1}]\cdots[\Sigma^{d-1} M_{1,d-1}]+\ldots +\lambda_r [M_{r,0}][\Sigma M_{r,1}]\cdots[\Sigma^{d-1} M_{r,d-1}]\\
=&\lambda_1 F_1 [M_{1,0}\oplus\Sigma M_{1,1}\oplus \cdots \oplus \Sigma^{d-1}M_{1,d-1}]+\hbox{smaller terms}=0,
\end{align*}
where $F_1\neq 0$ is the derived Hall number. Therefore $\lambda_1$ has to be zero. By induction on $r$, we obtain $\lambda_1=\ldots =\lambda_r=0$.
\end{proof}

To show that the set in Lemma~\ref{lemma-Hall-alg-basis-1} is a basis, we need a technical condition. A sequence $(f_i\colon M_i\to N_i)_{i\in\mathbb{Z}^+}$ of morphisms in $\ca$ is called a \emph{kernel-cokernel sequence} if $M_{i+1}=\ker(f_i)$ and $N_{i+1}=\cok(f_i)$ for all $i$. Note that such a sequence is stationary if and only if there exists $n\in\mathbb{Z}^+$ such that $f_i=0$ for all $i\geq n$.

\begin{lemma}\label{lemma-Hall-alg-basis-2} Assume that every kernel-cokernel sequence in $\ca$ is stationary. Then
the set $\{[M_0][\Sigma M_1] \cdots [\Sigma^{d-1} M_{d-1}]\mid M_0,M_1,\ldots,M_{d-1}\in\ca\}$ is a $\mathbb{Q}$-basis of $\ch(\cc)$.
\end{lemma}
\begin{proof} In view of Lemma~\ref{lemma-Hall-alg-basis-1}, it remains to show that the given set spans $\ch(\cc)$.
Let $N\in \mathcal{C}$ and write $N=N_0\oplus \Sigma N_1\oplus\ldots\oplus \Sigma^{d-1}N_{d-1}$. Without loss of generality, we may assume $N_0\neq 0$. Set $N'=\bigoplus\limits_{i=1}^{d-1} \Sigma^i N_i$, then $N\cong N_0\oplus N'$. We have
\begin{align*}
  [N_0][N']=F_{N_0,N'}^{N}[N]+\sum_{[Z]\neq [N]} F_{N_0,N'}^{Z}[Z]
\end{align*}
with $F_{N_0,N'}^{N}\neq 0$, and hence
\[
[N]=\frac{1}{F_{N_0,N'}^{N}}[N_0][N']-\sum_{[Z]\neq [N]} \frac{F_{N_0,N'}^{Z}}{F_{N_0,N'}^{N}}[Z].
\]
Consider a triangle
\[
\xymatrix{
N'\ar[r] & Z\ar[r] & N_0\ar[r]^h & \Sigma N'
}
\]
in $\mathcal{C}$. By Lemma~\ref{lem:when-is-root-category-left-homologically-finite} (a),
\begin{align*}
\Hom_\cc(N_0,\Sigma N')&=\Hom_\cc(N_0,\Sigma^2N_1\oplus\ldots\oplus \Sigma^{d-1}N_{d-2}\oplus \Sigma^d N_{d-1})\\
&=\Hom_\cc(N_0,\Sigma^d N_{d-1})=\Hom_\cc(N_0,\sigma N_{d-1})=\Hom_\ca(N_0,\sigma N_{d-1}).
\end{align*}
So we may consider $h$ as a column $(0,\ldots,0,f)^{\mathrm{tr}}$ with $f\in\Hom_\ca(N_0,\sigma N_{d-1})$. Complete $f$ to a triangle in $\cd$ and hence also in $\cc$
\[
\xymatrix{
\ker(f)\oplus\Sigma^{-1}\cok(f)\ar[r] & N_0\ar[r]^f\ar[r] &\sigma N_{d-1}\ar[r] & \Sigma\ker(f)\oplus \cok(f).
}
\]
Therefore
\begin{align*}
Z&\cong\ker(f)\oplus\Sigma N_1\oplus\ldots\oplus \Sigma^{d-2}N_{d-2}\oplus \Sigma^{-1}\cok(f)\\
&\cong\ker(f)\oplus\Sigma N_1\oplus\ldots\oplus \Sigma^{d-2}N_{d-2}\oplus \Sigma^{d-1}\sigma^{-1}\cok(f)
\end{align*}
 in $\cc$.
Repeating this procedure we obtain a kernel-cokernel sequence starting with $f\colon N_0\to \sigma N_{d-1}$, which is stationary by assumption. We obtain the desired result by induction.
\end{proof}

The notion of kernel-cokernel sequence is well-defined for general abelian categories $\ca$.
It is clear that if $\ca$ is artinian (that is, $\ca$ is skeletally small and every descending chain of subobjects of $\ca$ is stationary) or noetherian (that is, $\ca$ is skeletally small and every ascending chain of subobjects of $\ca$ is stationary), then every kernel-cokernel sequence in $\ca$ is stationary. Here are some examples:
\begin{itemize}
\item[(1)] If $Q$ is a quiver, then the category $\rep^b(Q)$ of finite-dimensional representations over $Q$ is both artinian and noetherian.
\item[(2)]$\rep^+(Q^l)$ is noetherian. In fact, it is equivalent to the category $\grmod k[x]$ of finitely generated graded $k[x]$-modules, where the degree of $x$ is 1.
\item[(3)] For a neotherian scheme $X$, the category $\coh(X)$ of coherent sheaves on $X$ is noetherian.
\item[(4)] Let $S$ be a totally ordered set such that for any $x\in S$, there are only finitely many $y\in S$ which are smaller than $x$. If there exists a map $\mathrm{l}\colon\mathrm{Ob}(\ca)\to S$ such that $\mathrm{l}(M)=\mathrm{l}(N)$ if $M\cong N$ and $\mathrm{l}(N)<\mathrm{l}(M)$ if $N$ is a proper subobject of $M$ (respectively, if $N$ is a proper quotient of $M$), then $\ca$ is artinian (respectively, noetherian).
\end{itemize}

\section{Derived Hall algebras of twisted root categories of $Q_{p_1,p_2}$ ($d\geq 2$)}
\label{s:derived-Hall-of-zigzag}

Let $p_1,p_2\geq 1$ be integers and $Q_{p_1,p_2}$ be the quiver with quiver-automorphism $\sigma=\sigma_{p_1,p_2}$ defined in the beginning of Section~\ref{s:quiver-generalised-zigzag-linear}. For integers $d\geq 1$ and $r\neq 0$, consider the $d$-root category of $\rep^b(Q_{p_1,p_2})$ twisted by $\sigma^{-r}$:
\[
\cc=\cd^b(\rep^b(Q_{p_1,p_2}))/\Sigma^d\circ\sigma^{-r}.
\]
In this section, we will give for $d\geq 2$ an explicit description of the derived Hall algebra $\ch(\cc)$ by providing a set of generators and relations. The case $d=1$ is exceptional due to the fact that the canonical functor $\ca\to\cc$ is not fully faithful, and we will deal with this case in a separate paper.

Let $\ch=\ch(p_1,p_2,d,r)$ be the $\mathbb{Q}$-algebra generated by $x_{i,j}$ ($i\in\mathbb{Z}$ and $j=0,1,\ldots,d-1$) subject to relations which will be given after the statement of Theorem~\ref{thm-for-zigzag}.

\begin{theorem}\label{thm-for-zigzag}
The assignment $x_{i,j}\mapsto [\Sigma^{j}S_i]$ ($i\in\mathbb{Z}$ and $j=0,1,\ldots,d-1$) defines an isomorphism $\ch\to\ch(\cc)$ of $\mathbb{Q}$-algebras.
\end{theorem}

Now we describe the relations.  Keep the notation in Section \ref{ss:quivers}: $n=p_1+p_2$, and for $i\in \mathbb{Z}$, $\overline{i}$ is the unique integer in the interval $[-p_1,p_2)$ such that $n|(i-\overline{i})$.  By convention we set $x_{i,d}=x_{\sigma(i),0}$. There are two cases:

\smallskip
\noindent Case $|nr|=2$ ($n=2$, $|r|=1$):
\begin{itemize}
   \item[(1)] relations for $x_{i,j}$ and $x_{i',j}$ ($i,i'\in\mathbb{Z}$ and $j=0,1,\ldots,d-1$):
   \begin{gather*}
   x_{i,j}x_{i+r,j}^2-q^{(-1)^{d+1}}(1+q^{-1})x_{i+r,j}x_{i,j}x_{i+r,j}+q^{2\cdot(-1)^{d+1}-1}x_{i+r,j}^2x_{i,j} ~\hbox{ for $\overline{i}=-1$},\\
    x_{i,j}^2x_{i+r,j}-q^{(-1)^{d+1}}(1+q^{-1})x_{i,j}x_{i+r,j}x_{i,j}+q^{2\cdot(-1)^{d+1}-1}x_{i+r,j}x_{i,j}^2 ~\hbox{ for $\overline{i}=-1$},\\
     x_{i,j}x_{i+r,j}^2-(1+q)x_{i+r,j}x_{i,j}x_{i+r,j}+qx_{i+r,j}^2x_{i,j} ~\hbox{ for $\overline{i}=0$},\\
    x_{i,j}^2x_{i+r,j}-(1+q)x_{i,j}x_{i+r,j}x_{i,j}+qx_{i+r,j}x_{i,j}^2 ~\hbox{ for $\overline{i}=0$},\\
    x_{i,j}x_{i+2kr,j}-q^{(-1)^{dk}}x_{i+2kr,j}x_{i,j} ~\hbox{ for } k>0,\\
    x_{i-(2k+1)r,j}x_{i,j}-x_{i,j}x_{i-(2k+1)r,j} ~\hbox{ for $\overline{i}=-1$ and $k>0$},\\
    x_{i-(2k+1)r,j}x_{i,j}-q^{(-1)^{d+1}-1}x_{i,j}x_{i-(2k+1)r,j} ~\hbox{ for $\overline{i}=0$ and $k>0$};    
   \end{gather*}

\item[(2)] relations for $x_{i,j+1}$ and $x_{i',j}$ ($i,i'\in\mathbb{Z}$ and $j=0,1,\ldots,d-1$):
  \begin{gather*}
  x_{i,j+1}x_{i-2kr,j}-q^{(-1)^{dk}}x_{i-2kr,j}x_{i,j+1}+\frac{\delta_{k,1}\delta_{d,2}}{q-1}q^{(-1)^{dk}} ~\hbox{ for } k>0,\\
    x_{i,j+1}x_{i-r,j}-q^{(-1)^{d+1}}x_{i-r,j}x_{i,j+1} ~\hbox{ for $\overline{i}=0$},\\
  x_{i,j+1}x_{i-(2k+1)r,j}-q^{(-1)^{d+1}-1}x_{i-(2k+1)r,j}x_{i,j+1} ~\hbox{ for $\overline{i}=0$ and $k>0$},\\
  x_{i,j+1}x_{i+2kr,j}-q^{(-1)^{dk+1}}x_{i+2kr,j}x_{i,j+1}-\frac{\delta_{k,0}}{q-1} ~\hbox{ for } k\geq0,\\
  x_{i,j+1}x_{i+(2k+1)r,j}-q^{(-1)^{d}+1}x_{i+(2k+1)r,j}x_{i,j+1} \hbox{ for $\overline{i}=-1$ and $k\geq 0$},\\
  x_{i,j+1}x_{i-r,j}- qx_{i-r,j}x_{i,j+1} ~\hbox{ for $\overline{i}=-1$},\\
  x_{i,j+1}x_{i',j}-x_{i',j}x_{i,j+1} ~\hbox{ for other  $i,i'$};
  \end{gather*}
  
\item[(3)] relations for $x_{i,j+t}$ and $x_{i',j}$ ($i,i'\in\mathbb{Z}$, $j=0,1,\ldots,d-1$ and $2\leq t\leq d-2$):
  \begin{gather*}
 x_{i,j+t}x_{i-2kr,j}-q^{(-1)^{dk+t-1}}x_{i-2kr,j}x_{i,j+t} ~\hbox{ for } k>0,\\
x_{i,j+t}x_{i-r,j}-q^{(-1)^{d+t}} x_{i-r,j}x_{i,j+t} ~\hbox{ for $\overline{i}=0$},\\
x_{i,j+t}x_{i-(2k+1)r,j}-q^{(-1)^{d+t}+(-1)^t} x_{i-(2k+1)r,j}x_{i,j+t} ~\hbox{ for $\overline{i}=0$ and $k\geq 0$},\\
  x_{i,j+t}x_{i+2kr,j}-q^{(-1)^{dk+t}}x_{i+2kr,j}x_{i,j+t} ~\hbox{ for } k\geq0,\\
  x_{i,j+t}x_{i+(2k+1)r,j}-q^{(-1)^{d+t-1}+(-1)^{t-1}}x_{i+(2k+1)r,j}x_{i,j+t} ~\hbox{ for $\overline{i}=-1$ and $k>0$},\\
  x_{i,j+t}x_{i-r,j}-q^{(-1)^{t-1}} x_{i-r,j}x_{i,j+t} ~\hbox{ for $\overline{i}=-1$},\\
  x_{i,j+t}x_{i',j}-x_{i',j}x_{i,j+t} ~\hbox{ for other  $i,i'$}.
  \end{gather*}
\end{itemize}

\smallskip
\noindent Case $|nr|>2$:
\begin{itemize}
   \item[(1)] relations for $x_{i,j}$ and $x_{i',j}$ ($i,i'\in\mathbb{Z}$ and $j=0,1,\ldots,d-1$):
   \begin{gather*}
   x_{i,j}x_{i+1,j}^2-(1+q^{-1})x_{i+1,j}x_{i,j}x_{i+1,j}+q^{-1}x_{i+1,j}^2x_{i,j} ~\hbox{ for $\overline{i}\in[-p_1,0)$},\\
   x_{i,j}^2x_{i+1,j}-(1+q^{-1})x_{i,j}x_{i+1,j}x_{i,j}+q^{-1}x_{i+1,j}x_{i,j}^2 ~\hbox{ for $\overline{i}\in[-p_1,0)$},\\
   x_{i,j}x_{i+1,j}^2-(1+q)x_{i+1,j}x_{i,j}x_{i+1,j}+qx_{i+1,j}^2x_{i,j} ~\hbox{ for $\overline{i}\in[0,p_2)$},\\
   x_{i,j}^2x_{i+1,j}-(1+q)x_{i,j}x_{i+1,j}x_{i,j}+qx_{i+1,j}x_{i,j}^2 ~\hbox{ for $\overline{i}\in[0,p_2)$},\\
   x_{i-nkr,j}x_{i,j}-q^{(-1)^{dk}}x_{i,j}x_{i-nkr,j} ~\hbox{ for } k>0,\\
   x_{i-(nkr+1),j}x_{i,j}-q^{(-1)^{dk+1}}x_{i,j}x_{i-(nkr+1),j} ~\hbox{ for $\overline{i}\in(-p_1,0]$ and $k>0$},\\
   x_{i-(nkr-1),j}x_{i,j}-q^{(-1)^{dk+1}}x_{i,j}x_{i-(nkr-1),j} ~\hbox{ for $\overline{i}\in[0,p_2)$ and $k>0$},\\
   x_{i,j}x_{i',j}-x_{i',j}x_{i,j} ~\hbox{ for other $i,i'$};
   \end{gather*}
  
\item[(2)] relations for $x_{i,j+1}$ and $x_{i',j}$ ($i,i'\in\mathbb{Z}$ and $j=0,1,\ldots,d-1$):
  \begin{gather*}
 	x_{i,j+1}x_{i-nkr,j}-q^{(-1)^{dk}}x_{i-nkr,j}x_{i,j+1}+\frac{\delta_{k,1}\delta_{d,2}}{q-1}q^{(-1)^{dk}} ~\hbox{ for } k>0,\\
	x_{i,j+1}x_{i-(nkr+1),j}-q^{(-1)^{dk+1}}x_{i-(nkr+1),j}x_{i,j+1} ~\hbox{ for $\overline{i}\in(-p_1,0]$ and $k>0$},\\
	x_{i,j+1}x_{i-(nkr-1),j}-q^{(-1)^{dk+1}}x_{i-(nkr-1),j}x_{i,j+1} ~\hbox{ for $\overline{i}\in[0,p_2)$ and $k>0$},\\
  	x_{i,j+1}x_{i+nkr,j}-q^{(-1)^{dk+1}}x_{i+nkr,j}x_{i,j+1}-\frac{\delta_{k,0}}{q-1} ~\hbox{ for } k\geq0,\\
  	x_{i,j+1}x_{i+(nkr+1),j}-q^{(-1)^{dk}}x_{i+(nkr+1),j}x_{i,j+1} ~\hbox{ for $\overline{i}\in[-p_1,0)$ and $k\geq0$},\\
  	x_{i,j+1}x_{i+(nkr-1),j}-q^{(-1)^{dk}}x_{i+(nkr-1),j}x_{i,j+1} ~\hbox{ for $\overline{i}\in(0,p_2]$ and $k\geq0$},\\
  	x_{i,j+1}x_{i',j}-x_{i',j}x_{i,j+1} ~\hbox{ for other }i,i';
  \end{gather*}

\item[(3)]  relations for $x_{i,j+t}$ and $x_{i',j}$ ($i\in\mathbb{Z}$, $j=0,1,\ldots,d-1$ and $2\leq t\leq d-2$):
  \begin{equation*}
   x_{i,j+t}x_{i-nkr,j}-q^{(-1)^{dk+t-1}}x_{i-nkr,j}x_{i,j+t} ~\hbox{ for } k>0
  \end{equation*}
  \begin{equation*}
  x_{i,j+t}x_{i-(nkr+1),j}-q^{(-1)^{dk+t}}x_{i-(nkr+1),j}x_{i,j+t} ~\hbox{ for $\overline{i}\in(-p_1,0]$ and $k>0$}
  \end{equation*}
  \begin{equation*}
   x_{i,j+t}x_{i-(nkr-1),j}-q^{(-1)^{dk+t}}x_{i-(nkr-1),j}x_{i,j+t} ~\hbox{ for $\overline{i}\in[0,p_2)$ and $k>0$}
  \end{equation*}
  \begin{equation*}
    x_{i,j+t}x_{i+nkr,j}-q^{(-1)^{dk+t}}x_{i+nkr,j}x_{i,j+1} ~\hbox{ for } k\geq0
  \end{equation*}
  \begin{equation*}
    x_{i,j+t}x_{i+(nkr+1),j}-q^{(-1)^{dk+t-1}}x_{i+(nkr+1),j}x_{i,j+t} ~\hbox{ for $\overline{i}\in[-p_1,0)$ and $k\geq0$}
  \end{equation*}
  \begin{equation*}
    x_{i,j+t}x_{i+(nkr-1),j}-q^{(-1)^{dk+t-1}}x_{i+(nkr-1),j}x_{i,j+t} ~\hbox{ for $\overline{i}\in(0,p_2]$ and $k\geq0$}
  \end{equation*}
  \begin{equation*}
   x_{i,j+t}x_{i',j}-x_{i',j}x_{i,j+t} ~\hbox{ for other }i,i'.
  \end{equation*}
\end{itemize}

Note that the relations in group (3) are empty for $d=2$ and $d=3$. The rest of this section is devoted to the proof of Theorem~\ref{thm-for-zigzag}, which is divided into three steps.

First, we verify the relations in group (1). For $j=0,1,\ldots,d-1$, let $\ch_j(\cc)$ be the subspace of $\ch(\cc)$ spanned by $[\Sigma^j M]$, $M\in\rep^b(Q_{p_1,p_2})$. Then it is a subalgebra of $\ch(\cc)$, by Proposition~\ref{prop:Hall-subalgebra-in-root-category} (a). Let $\ch_j$ be the subalgebra of $\ch$ generated by $x_{ij}$ ($i\in\mathbb{Z}$).

\begin{lemma}
\label{lem:relations-in-the-first-group-zigzag}
Let $j=0,1,\ldots,d-1$.
The assignment $x_{i,j}\mapsto [\Sigma^j S_i]$ ($i\in\mathbb{Z}$) defines an isomorphism $\ch_j\to\ch_j(\cc)$ of $\mathbb{Q}$-algebras.
\end{lemma}
\begin{proof}
It is enough to prove the statement for the case $j=0$. By Lemma~\ref{lem:RH-algebra-of-quiver-with-zigzag-orientation}, the Ringel--Hall algebra $\ch(\rep^b(Q_{p_1,p_2}))$ of $\rep^b(Q_{p_1,p_2})$ is generated by $\{[S_i]\}_{i\in \mathbb{Z}}$ subject to the following relations:
\begin{equation*}
    [S_i]\diamond[S_{i+1}]^{\diamond 2}-(1+q^{-1})[S_{i+1}]\diamond[S_i]\diamond[S_{i+1}]+q^{-1}[S_{i+1}]^{\diamond2}\diamond[S_i]=0 ~\hbox{ for $\overline{i}\in[-p_1,0)$},
\end{equation*}
\begin{equation*}
   [S_i]^{\diamond2}\diamond[S_{i+1}]-(1+q^{-1})[S_i]\diamond[S_{i+1}]\diamond[S_i]+q^{-1}[S_{i+1}]\diamond[S_i]^{\diamond2} =0~\hbox{ for $\overline{i}\in[-p_1,0)$},
\end{equation*}
\begin{equation*}
    [S_i]\diamond[S_{i+1}]^{\diamond 2}-(1+q)[S_{i+1}]\diamond[S_i]\diamond[S_{i+1}]+q[S_{i+1}]^{\diamond2}\diamond[S_i]=0 ~\hbox{ for $\overline{i}\in[0,p_2)$},
  \end{equation*}
  \begin{equation*}
   [S_i]^{\diamond2}\diamond[S_{i+1}]-(1+q)[S_i]\diamond[S_{i+1}]\diamond[S_i]+q[S_{i+1}]\diamond[S_i]^{\diamond2}=0 ~\hbox{ for $\overline{i}\in[0,p_2)$},
  \end{equation*}
  \begin{equation*}
  [S_i]\diamond[S_{i'}]-[S_{i'}]\diamond[S_i]=0 ~\hbox{ if }|i-i'|>1.
  \end{equation*}
This, by Proposition~\ref{prop:Hall-subalgebra-in-root-category} (a), implies that $\ch_0(\cc)$ is generated by $\{[S_i]\}_{i\in \mathbb{Z}}$ subject to relations which are obtained from the above relations by twisting with $\langle-,-\rangle_o$.
For example,
\begin{align*}
  [S_i][S_{i+1}]^2 &= q^{2\langle S_{i+1}, S_i\rangle_o+\langle S_{i+1},S_{i+1}\rangle_o}[S_i]\diamond[S_{i+1}]^{\diamond 2}=q^{2\langle S_{i+1}, S_i\rangle_o}[S_i]\diamond[S_{i+1}]^{\diamond 2} \\
  {[S_{i+1}][S_i][S_{i+1}]} &= q^{\langle S_i,S_{i+1}\rangle_o+\langle S_{i+1},S_{i}\rangle_o+\langle S_{i+1},S_{i+1}\rangle_o}{[S_{i+1}]\diamond[S_i]\diamond[S_{i+1}]}\\
  &=q^{\langle S_i,S_{i+1}\rangle_o+\langle S_{i+1},S_{i}\rangle_o}{[S_{i+1}]\diamond[S_i]\diamond[S_{i+1}]} \\
{[S_{i+1}]^2[S_i]} &= q^{2\langle S_{i}, S_{i+1}\rangle_o+\langle S_{i+1},S_{i
  	+1}\rangle_o}{[S_{i+1}]^{\diamond2}\diamond[S_i]}= q^{2\langle S_{i}, S_{i+1}\rangle_o}{[S_{i+1}]^{\diamond2}\diamond[S_i]},
\end{align*}
because $\langle S_{i+1},S_{i+1}\rangle_o=0$ for any $i\in\mathbb{Z}$, by Lemma~\ref{lem:orbital-Euler-form-zigzag}.
Therefore,
\begin{equation*}
[S_i][S_{i+1}]^{2}-(1+q^{-1})q^{(S_{i+1},S_i)_o}[S_{i+1}][S_i][S_{i+1}]+q^{-1+2(S_{i+1},S_i)_o}[S_{i+1}]^{2}[S_i]=0~\text{if }\overline{i}\in[-p_1,0),
\end{equation*}
 \begin{equation*}
[S_i][S_{i+1}]^{2}-(1+q)q^{(S_{i+1},S_i)_o}[S_{i+1}][S_i][S_{i+1}]+q^{1+2(S_{i+1},S_i)_o}[S_{i+1}]^{2}[S_i]=0~\text{if }\overline{i}\in[0,p_2).
\end{equation*}
The desired relations then follow from Corollary~\ref{cor:antisymmetric-orbital-Euler-form-zigzag}. The other two relations of degree $3$ are obtained similarly.

To verify the relations of degree $2$, we have by Proposition~\ref{prop:Hall-subalgebra-in-root-category} (a)
\begin{align*}
[S_i][S_{i'}]=q^{\langle S_{i'},S_i\rangle_o}[S_i]\diamond[S_{i'}],~~
[S_{i'}][S_i]=q^{\langle S_{i},S_{i'}\rangle_o}[S_{i'}]\diamond[S_{i}].
\end{align*}
So
\begin{align*}
[S_i][S_{i'}]-q^{(S_{i'},S_{i})_o}[S_{i'}][S_i]=0.
\end{align*}
The desired relations then follow from Corollary~\ref{cor:antisymmetric-orbital-Euler-form-zigzag}.
\end{proof}

Secondly, we verify the relations in groups (2) and (3). The relations in group (2) follow from Corollary~\ref{cor:commutator-for-neighbouring-pieces-simples}, Lemma~\ref{lem:Euler-form-zigzag} and Corollary~\ref{cor:antisymmetric-orbital-Euler-form-zigzag}. The relations in group (3) directly follow from Corollary~\ref{cor:commutator-for-pieces-faraway} and Lemma~\ref{lem:Euler-form-zigzag} and Corollary~\ref{cor:antisymmetric-orbital-Euler-form-zigzag}.

\medskip
Finally, by the results in the preceding two steps, the assignment $x_{i,j}\mapsto [\Sigma^{j}S_i]$ ($i\in\mathbb{Z}$ and $j=0,1,\ldots,d-1$) defines a homomorphism $\ch\to\ch(\cc)$ of $\mathbb{Q}$-algebras, which we denote by $\varphi$.
Notice that groups (2) and (3) consist of $q$-commutative relations for $x_{i,j}$ and $x_{i',j'}$,where $i,i'\in\mathbb{Z}$ and $j,j'=0,1,\ldots,d-1$ with $j\neq j'$, and there is exactly one such relation for each pair $\{x_{i,j},x_{i',j'}\}$. It follows that
\[
\ch=\ch_0\otimes_{\mathbb{Q}}\ch_1\otimes_{\mathbb{Q}}\cdots\otimes_{\mathbb{Q}}\ch_{d-1}
\] as $\mathbb{Q}$-vector space. By Lemma~\ref{lemma-Hall-alg-basis-2}, we have
\[
\ch(\cc)=\ch_0(\cc)\otimes_\mathbb{Q}\ch_1(\cc)\otimes_\mathbb{Q}\cdots\otimes_\mathbb{Q}\ch_{d-1}(\cc)
\]
as $\mathbb{Q}$-vector space. Moreover, by Lemma~\ref{lem:relations-in-the-first-group-zigzag}, $\varphi$ restricts to isomorphisms $\ch_j\to\ch_j(\cc)$ for all $j=0,1,\ldots,d-1$, and hence it is an isomorphism. The proof of Theorem~\ref{thm-for-zigzag} is finished.

\section{Derived Hall algebras of twisted root categories of $Q^{l}$ ($d\geq 2$)}
\label{s:derived-Hall-of-linear}

Let $Q^{l}$ be the quiver with quiver-automorphism $\sigma=\sigma_{0,1}$ defined in the beginning of Section~\ref{s:quiver-generalised-zigzag-linear}. For integers $d\geq 1$ and $r\neq 0$, consider the $d$-root categories of $\rep^b(Q^{l})$ and $\rep^+(Q^{l})$ twisted by $\sigma^{-r}$:
\[
\cc=\cd^b(\rep^b(Q^{l}))/\Sigma^d\circ\sigma^{-r}, \text{and  }\cc^+=\cd^b(\rep^+(Q^{l}))/\Sigma^d\circ\sigma^{-r}.
\]
In this section, we will give for $d\geq 2$ an explicit description of the derived Hall algebras $\ch(\cc)$ and $\ch(\cc^+)$ by providing sets of generators and relations. The case $d=1$ is exceptional due to the fact that the canonical functors $\rep^b(Q^{l})\to\cc$ and $\rep^+(Q^{l})\to\cc^+$ are not fully faithful, and we will deal with this case in a separate paper.

\subsection{The derived Hall algebra $\ch(\cc)$}
\label{ss:derived-Hall-of-linear-bounded}
Let $\ch=\ch(0,1,d,r)$ be the $\mathbb{Q}$-algebra generated by $x_{i,j}$ ($i\in\mathbb{Z}$ and $j=0,1,\ldots,d-1$) subject to relations which will be given after the statement of Theorem~\ref{thm-for-linear}.

\begin{theorem}\label{thm-for-linear}
The assignment $x_{i,j}\mapsto [\Sigma^{j}S_i]$ ($i\in\mathbb{Z}$ and $j=0,1,\ldots,d-1$) defines an isomorphism $\ch\to\ch(\cc)$ of $\mathbb{Q}$-algebras.
\end{theorem}

Now we describe the relations. By convention we set $x_{i,d}=x_{\sigma(i),0}$. There are three cases:

\smallskip
\noindent Case $|r|=1$:
\begin{itemize}
   \item[(1)] the relations for $x_{i,j}$ and $x_{i',j}$ ($i,i'\in\mathbb{Z}$ and $j=0,1,\ldots,d-1$):
   \begin{gather*}
   	x_{i,j}x_{i+1,j}^2-q^{r\cdot[(-1)^d-1]}(1+q^{r})x_{i+1,j}x_{i,j}x_{i+1,j}+q^{r\cdot[2\cdot(-1)^d-1]}x_{i+1,j}^2x_{i,j},\\
   	x_{i,j}^2x_{i+1,j}-q^{r\cdot[(-1)^d-1]}(1+q^{r})x_{i,j}x_{i+1,j}x_{i,j}+q^{r\cdot[2\cdot(-1)^d-1]}x_{i+1,j}x_{i,j}^2,\\
   	x_{i,j}x_{i+rk,j}-q^{(-1)^{k}[(-1)^{d+1}+1]}x_{i+rk,j}x_{i,j} ~\hbox{ for $k>1$;}
   \end{gather*}

 \item[(2)] the relations for $x_{i,j+1}$ and $x_{i',j}$ ($i,i'\in\mathbb{Z}$ and $j=0,1,\ldots,d-1$):
 \begin{gather*}
 	x_{i,j+1}x_{i-r,j}-q^{(-1)^{d}}x_{i-r,j}x_{i,j+1}+\frac{\delta_{d,2}\cdot q^{\frac{r+1}{2}}}{q-1}q^{(-1)^{d}},\\
 	x_{i,j+1}x_{i-kr,j}-q^{(-1)^{k}[(-1)^{d+1}+1]}x_{i-kr,j}x_{i,j+1}+\frac{\delta_{d,2}\cdot q^{(-1)^{k}[(-1)^{d+1}+1]}}{q-1}~\hbox{ for $k>1$},\\
 	x_{i,j+1}x_{i,j}-q^{-1}x_{i,j}x_{i,j+1}-\frac{q^{\frac{r+1}{2}\cdot(-1)^d}}{q-1},\\
 	x_{i,j+1}x_{i+kr,j}-q^{(-1)^{k}[(-1)^{d}-1]}x_{i+kr,j}x_{i,j+1} ~\hbox{ for $k>0$};
 \end{gather*}
   
  \item[(3)] the relations for $x_{i,j+t}$ and $x_{i',j}$ ($i,i'\in\mathbb{Z}$, $j=0,1,\ldots,d-1$ and $2\leq t\leq d-2$):
  \begin{gather*}
  	x_{i,j+t}x_{i-r,j}-q^{(-1)^{d+t-1}}x_{i-r,j}x_{i,j+t},\\
  	x_{i,j+t}x_{i-kr,j}-q^{(-1)^{k+t-1}[(-1)^{d+1}+1]}x_{i-kr,j}x_{i,j+t} ~\hbox{ for $k>1$},\\
  	x_{i,j+t}x_{i,j}-q^{(-1)^{t}}x_{i,j}x_{i,j+t} ~\hbox{ for $k\geq0$},\\
  	x_{i,j+t}x_{i+kr,j}-q^{(-1)^{k+t-1}[(-1)^{d}-1]}x_{i+kr,j}x_{i,j+t} ~\hbox{ for $k>0$}.
  \end{gather*}
\end{itemize}

\smallskip
\noindent Case $r=2$:
\begin{itemize}
   \item[(1)] the relations for $x_{i,j}$ and $x_{i',j}$ ($i,i'\in\mathbb{Z}$ and $j=0,1,\ldots,d-1$):
  \begin{gather*}
    x_{i,j}x_{i+1,j}^2-q^{(-1)^{d+1}}(1+q)x_{i+1,j}x_{i,j}x_{i+1,j}+q^{2\cdot(-1)^{d+1}+1}x_{i+1,j}^2x_{i,j},\\  
   x_{i,j}^2x_{i+1,j}-q^{(-1)^{d+1}}(1+q)x_{i,j}x_{i+1,j}x_{i,j}+q^{2\cdot(-1)^{d+1}+1}x_{i+1,j}x_{i,j}^2,\\
   x_{i,j}x_{i+(2k-1),j}-q^{(-1)^{dk+1}}x_{i+(2k-1),j}x_{i,j} ~\hbox{ for }k>1,\\
    x_{i,j}x_{i+2k,j}-q^{(-1)^{dk}}x_{i+2k,j}x_{i,j} ~\hbox{ for }k>0;
  \end{gather*}

\item[(2)] the relations for $x_{i,j+1}$ and $x_{i',j}$ ($i,i'\in\mathbb{Z}$ and $j=0,1,\ldots,d-1$):
  \begin{gather*}
  x_{i,j+1}x_{i-2k,j}-q^{(-1)^{dk}}x_{i-2k,j}x_{i,j+1}+\frac{\delta_{d,2}\delta_{k,0}}{q-1}q^{(-1)^{dk}} ~\hbox{ for } k>0,\\
  x_{i,j+1}x_{i-(2k-1),j}-q^{(-1)^{dk+1}+\delta_{k,1}} x_{i-(2k-1),j}x_{i,j+1}~\hbox{ for } k>0,\\
    x_{i,j+1}x_{i+2k,j}-q^{(-1)^{dk+1}}x_{i+2k,j}x_{i,j+1}-\frac{\delta_{k,0}}{q-1} ~\hbox{ for } k\geq0,\\
    x_{i,j+1}x_{i+(2k-1),j}-q^{(-1)^{dk}}x_{i+(2k-1),j}x_{i,j+1} ~\hbox{ for } k>0;
  \end{gather*}

\item[(3)] the relations for $x_{i,j+t}$ and $x_{i',j}$ ($i\in\mathbb{Z}$, $j=0,1,\ldots,d-1$ and $2\leq t\leq d-2$):
  \begin{gather*}
 x_{i,j+t}x_{i-2k,j}-q^{(-1)^{dk+t-1}}x_{i-2k,j}x_{i,j+t} ~\hbox{ for } k>0,\\
  x_{i,j+t}x_{i-(2k-1),j}-q^{(-1)^{t-1}\cdot[(-1)^{dk+1}+\delta_{k,1}]}x_{i-(2k-1),j}x_{i,j+t} ~\hbox{ for } k>0,\\
    x_{i,j+t}x_{i+2k,j}-q^{(-1)^{dk+t}}x_{i+2k,j}x_{i,j+t} ~\hbox{ for } k\geq0,\\
    x_{i,j+t}x_{i+(2k-1),j}-q^{(-1)^{dk+t-1}}x_{i+(2k-1),j}x_{i,j+t} ~\hbox{ for } k>0.
  \end{gather*}

\end{itemize}

\smallskip
\noindent Case $|r|>1$ but $r\neq2$:
\begin{itemize}
   \item[(1)] the relations for $x_{i,j}$ and $x_{i',j}$ ($i,i'\in\mathbb{Z}$ and $j=0,1,\ldots,d-1$):
  \begin{equation*}
    x_{i,j}x_{i+1,j}^2-(1+q)x_{i+1,j}x_{i,j}x_{i+1,j}+qx_{i+1,j}^2x_{i,j},
  \end{equation*}
  \begin{equation*}
   x_{i,j}^2x_{i+1,j}-(1+q)x_{i,j}x_{i+1,j}x_{i,j}+qx_{i+1,j}x_{i,j}^2,
  \end{equation*}
  \begin{equation*}
    x_{i,j}x_{i+(kr-1),j}-q^{(-1)^{dk+1}}x_{i+(kr-1),j}x_{i,j} ~\hbox{ for }k>0,
  \end{equation*}
  \begin{equation*}
    x_{i,j}x_{i+kr,j}-q^{(-1)^{dk}}x_{i+kr,j}x_{i,j} ~\hbox{ for }k>0,
  \end{equation*}
  \begin{equation}\label{relation:commutative-in-the-same-piece}
    x_{i,j}x_{i',j}-x_{i',j}x_{i,j} ~\hbox{ for other $i, i'$};
  \end{equation}

\item[(2)] the relations for $x_{i,j+1}$ and $x_{i',j}$ ($i,i'\in\mathbb{Z}$ and $j=0,1,\ldots,d-1$):
  \begin{equation*}
  x_{i,j+1}x_{i-kr,j}-q^{(-1)^{dk}}x_{i-kr,j}x_{i,j+1}+\frac{\delta_{d,2}\delta_{k,0}}{q-1}q^{(-1)^{dk}} ~\hbox{ for } k>0,
  \end{equation*}
    \begin{equation*}
    x_{i,j+1}x_{i-(kr-1),j}-q^{(-1)^{dk+1}\cdot (1-\delta_{k,0})}x_{i-(kr-1),j}x_{i,j+1} ~\hbox{ for } k\geq0,
  \end{equation*}
  \begin{equation*}
    x_{i,j+1}x_{i+kr,j}-q^{(-1)^{dk+1}}x_{i+kr,j}x_{i,j+1}-\frac{\delta_{k,0}}{q-1} ~\hbox{ for } k\geq0,
  \end{equation*}
  \begin{equation*}
   x_{i,j+1}x_{i+(kr-1),j}-q^{(-1)^{dk}}x_{i+(kr-1),j}x_{i,j+1} ~\hbox{ for } k\geq0,
  \end{equation*}
  \begin{equation}\label{relation:commutative-in-the-neighbour-pieces}
    x_{i,j+1}x_{i',j}-x_{i',j}x_{i,j+1} ~\hbox{ for other $i, i'$ };
  \end{equation}

\item[(3)] the relations for $x_{i,j+t}$ and $x_{i',j}$ ($i\in\mathbb{Z}$, $j=0,1,\ldots,d-1$ and $2\leq t\leq d-2$):
  \begin{equation*}
  x_{i,j+t}x_{i-kr,j}-q^{(-1)^{dk+t-1}}x_{i-kr,j}x_{i,j+t} ~\hbox{ for } k>0,
  \end{equation*}
    \begin{equation*}
   x_{i,j+t}x_{i-(kr-1),j}-q^{(-1)^{dk+t}\cdot (1-\delta_{k,0})}x_{i-(kr-1),j}x_{i,j+t}~\hbox{ for } k\geq0,
  \end{equation*}
  \begin{equation*}
    x_{i,j+t}x_{i+kr,j}-q^{(-1)^{dk+t}}x_{i+kr,j}x_{i,j+t} ~\hbox{ for } k\geq0,
  \end{equation*}
   \begin{equation*}
   x_{i,j+t}x_{i+(kr-1),j}-q^{(-1)^{dk+t-1}}x_{i+(kr-1),j}x_{i,j+t} ~\hbox{ for } k\geq0,
  \end{equation*}
  \begin{equation}\label{relation:commutative-in-the-far-pieces}
    x_{i,j+1}x_{i',j}-x_{i',j}x_{i,j+1} ~\hbox{ for other $i, i'$ }.
  \end{equation}
\end{itemize}

Note that if $r=-2$, the relations (\ref{relation:commutative-in-the-same-piece}), (\ref{relation:commutative-in-the-neighbour-pieces}) and (\ref{relation:commutative-in-the-far-pieces}) disappear. Moreover, the relations in group (3) are empty for $d=2$ and $d=3$. The strategy of the proof of Theorem~\ref{thm-for-linear} is the same as that of Theorem~\ref{thm-for-zigzag}.
For $j=0,1,\ldots,d-1$, let $\ch_j(\cc)$ be the subspace of $\ch(\cc)$ spanned by $[\Sigma^j M]$, $M\in\rep^b(Q^l)$. Then it is a subalgebra of $\ch(\cc)$, by Proposition~\ref{prop:Hall-subalgebra-in-root-category} (a). Let $\ch_j$ be the subalgebra of $\ch$ generated by $x_{ij}$ ($i\in\mathbb{Z}$).

\begin{lemma}
\label{lem:relations-in-the-first-group-linear}
Let $j=0,1,\ldots,d-1$.
The assignment $x_{i,j}\mapsto [\Sigma^j S_i]$ ($i\in\mathbb{Z}$) defines an isomorphism $\ch_j\to\ch_j(\cc)$ of $\mathbb{Q}$-algebras.
\end{lemma}

This lemma verifies the relations in group (1). Its proof is similar to that of Lemma~\ref{lem:relations-in-the-first-group-zigzag}. In particular, the isoclass of the indecomposable representation $V_{a,b}$, as an element of $\ch_0(\cc)$, can be represented by a non-commutative polynomial of the $[S_i]'s$. Let $\Phi_{a,b}(x_i|i\in\mathbb{Z})$ be the non-commutative polynomial in variables $x_i,i\in\mathbb{Z}$, such that 
\begin{equation}\label{equation-for-indecomposable}
	[V_{a,b}]=\Phi_{a,b}([S_i]|i\in\mathbb{Z}).
\end{equation}
This polynomial will be used in Section~\ref{ss:derived-Hall-of-linear-$+$}.

The relations in group (2) follow from Corollary~\ref{cor:commutator-for-neighbouring-pieces-simples}, Lemma~\ref{lem:Euler-form-linear} and Corollary~\ref{cor:antisymmetric-orbital-Euler-form-linear}. The relations in group (3) directly follow from Corollary~\ref{cor:commutator-for-pieces-faraway} and Lemma~\ref{lem:Euler-form-linear} and Corollary~\ref{cor:antisymmetric-orbital-Euler-form-linear}.  
Therefore, the assignment $x_{i,j}\mapsto [\Sigma^{j}S_i]$ ($i\in\mathbb{Z}$ and $j=0,1,\ldots,d-1$) defines a homomorphism $\ch\to\ch(\cc)$ of $\mathbb{Q}$-algebras, which we denote by $\varphi$. Moreover, 
\[
\ch=\ch_0\otimes_{\mathbb{Q}}\ch_1\otimes_{\mathbb{Q}}\cdots\otimes_{\mathbb{Q}}\ch_{d-1}
~~\text{and}~~
\ch(\cc)=\ch_0(\cc)\otimes_\mathbb{Q}\ch_1(\cc)\otimes_\mathbb{Q}\cdots\otimes_\mathbb{Q}\ch_{d-1}(\cc)
\]
as $\mathbb{Q}$-vector spaces.  Since $\varphi$ restricts to isomorphisms $\ch_j\to\ch_j(\cc)$ for all $j=0,1,\ldots,d-1$ by Lemma~\ref{lem:relations-in-the-first-group-linear}, it follows that $\varphi$ is an isomorphism.

\subsection{The derived Hall algebra $\ch(\cc^+)$}
\label{ss:derived-Hall-of-linear-$+$}

It follows by Lemma~\ref{lem:when-is-root-category-left-homologically-finite} (c) that the following three conditions are equivalent:
\begin{itemize}
\item[-] $\cc^+$ is left locally homologically finite,
\item[-] $\sigma^r$ satisfies the condition (HF),
\item[-] $r>0$. 
\end{itemize}
So in the rest of this subsection we assume $r>0$.

Let $\ch^+=\ch^+(0,1,d,r)$ be the $\mathbb{Q}$-algebra generated by $x_{i,j}$ and $z_{i,j}$ ($i\in\mathbb{Z}$ and $j=0,1,\ldots,d-1$) subject to relations which will be given after the statement of Theorem~\ref{thm-for-linear-$+$}.

\begin{theorem}\label{thm-for-linear-$+$}
	The assignment $x_{i,j}\mapsto [\Sigma^{j}S_i]$, $z_{i,j}\mapsto [\Sigma^{j}P_i]$ ($i\in\mathbb{Z}$ and $j=0,1,\ldots,d-1$) defines an isomorphism $\ch^+\to\ch(\cc^+)$ of $\mathbb{Q}$-algebras.
\end{theorem}

Now we describe the relations. By convention we set $x_{i,d}=x_{\sigma(i),0}$ and $z_{i,d}=z_{\sigma(i),0}$. The variables $\{x_{ij}\mid i\in\mathbb{Z},j=0,1,\ldots,d-1\}$ generate $\ch=\ch(0,1,d,r)$ as a subalgebra of $\ch^+$. For the remaining relations there are also three cases (all of the following $k$'s are positive without further remark):

\smallskip
\noindent Case $r=1$:
\begin{itemize}
	\item[(1)] the relations for $z_{i,j}$ and $x_{i',j}$, and for $z_{i,j}$ and $z_{i',j}$ ($i,i'\in\mathbb{Z}$ and $j=0,1,\ldots,d-1$):
	\begin{equation}\label{relation:projective-by-simple-and-projective-$r=1$}
		x_{i-1,j}z_{i,j}-q^{(-1)^{d}}z_{i,j}x_{i-1,j}-q^{(-1)^d}z_{i-1} 
	\end{equation}
	\begin{equation*}
		x_{i-k,j}z_{i,j}-q^{(-1)^{dk}}z_{i,j}x_{i-k,j} ~\hbox{ for $k>1$}
	\end{equation*}
	\begin{equation*}
		x_{i,j}z_{i,j}-q^{(-1)^{d}+1}z_{i,j}x_{i,j} 
	\end{equation*}
	\begin{equation*}
		x_{i+k-1,j}z_{i,j}-q^{(-1)^{dk}}z_{i,j}x_{i+k-1,j} ~\hbox{ for $k>1$.}
	\end{equation*}
	\begin{equation*}
		z_{i,j}z_{i-k,j}-q^{-1-\sum\limits_{p=1}^{k}(-1)^{dp} }z_{i-k,j}z_{i,j}
	\end{equation*}
	\begin{equation*}
		z_{i,j}z_{i+k,j}-q^{1+\sum\limits_{p=1}^{k}(-1)^{dp} }z_{i+k,j}z_{i,j}
	\end{equation*}

	\item[(2)] the relations for $z_{i,j+1}$ and $z_{i',j}$, for $z_{i,j+1}$ and $x_{i',j}$, and for $x_{i',j+1}$ and $z_{i,j}$ ($i,i'\in\mathbb{Z}$ and $j=0,1,\ldots,d-1$):
	\begin{equation*}
		z_{i,j+1}z_{i-k,j}-q^{\sum\limits_{p=1}^{k}(-1)^{dp} }z_{i-k,j}z_{i,j+1}-\Phi_{i-k,i-1}(x_{\alpha,j})_{\alpha\in\mathbb{Z}}+\delta_{d,2}\delta_{i,i-k+1}\frac{q^{(-1)^d}}{q-1}
	\end{equation*}
	\begin{equation}\label{relation:simple-by-projectives-$r=1$}
		z_{i,j+1}z_{i,j}-q^{-1}z_{i,j}z_{i,j+1}-\frac{1}{q-1}+\delta_{d,2}q^{-1}x_{i,j+1}
	\end{equation}
	\begin{equation*}
		z_{i,j+1}z_{i+k,j}-q^{-1+\sum\limits_{p=1}^{k}(-1)^{dp+1}}z_{i+k,j}z_{i,j+1}+\delta_{d,2}q^{-1+\sum\limits_{p=1}^{k}(-1)^{dp+1}}\Phi_{i,i+k}(x_{\alpha,j+1})_{\alpha\in\mathbb{Z}}
	\end{equation*}
	\begin{equation*}
		z_{i,j+1}x_{i-k,j}-q^{(-1)^{dk}}x_{i-k,j}z_{i,j+1}
		~\hbox{ for $k>1$}
	\end{equation*}
	\begin{equation*}
		z_{i,j+1}x_{i-1,j}-q^{(-1)^{d}+1 }x_{i-1,j}z_{i,j+1}
	\end{equation*}
	\begin{equation*}
		z_{i,j+1}x_{i,j}-q^{(-1)^{d}}x_{i,j}z_{i,j+1}-q^{(-1)^d}z_{i+1,j+1}
	\end{equation*}
	\begin{equation*}
		z_{i,j+1}x_{i+k-1,j}-q^{(-1)^{dk}}x_{i+k-1,j}z_{i,j+1}~\hbox{ for $k>1$}
	\end{equation*}
	\begin{equation*}
		x_{i,j+1}z_{i+k,j}-q^{(-1)^{dk+1}}z_{i+k,j}x_{i,j+1}
	\end{equation*}
	\begin{equation*}
		x_{i,j+1}z_{i,j}-q^{(-1)^{d+1}-1}z_{i,j}x_{i,j+1} 
	\end{equation*}
	\begin{equation*}
		x_{i,j+1}z_{i-1,j}-q^{-1}z_{i-1,j}x_{i,j+1}+\delta_{d,2}z_{i,j}
	\end{equation*}
	\begin{equation*}
		x_{i,j+1}z_{i-k+1,j}-q^{(-1)^{dk+1}}z_{i-k+1,j}x_{i,j+1} ~\hbox{ for $k>2$}
	\end{equation*}

	\item[(3)] the relations for $z_{i,j+t}$ and $z_{i',j}$, for $z_{i,j+t}$ and $x_{i',j}$, and for $x_{i',j+t}$ and $z_{i,j}$ ($i,i'\in\mathbb{Z}$ and $j=0,1,\ldots,d-1$):
	\begin{equation*}
		z_{i,j+t}z_{i-k,j}-q^{\sum\limits_{p=1}^{k}(-1)^{dp+t-1}}z_{i-k,j}z_{i,j+t}
	\end{equation*}
	\begin{equation*}
		z_{i,j+t}z_{i,j}-q^{(-1)^t}z_{i,j}z_{i,j+t}
	\end{equation*}
	\begin{equation*}
		z_{i,j+t}z_{i+k,j}-q^{(-1)^t+\sum\limits_{p=1}^{k}(-1)^{dp+t}}z_{i+k,j}z_{i,j+t}
	\end{equation*}
	\begin{equation*}
		z_{i,j+t}x_{i-k,j}-q^{(-1)^{dk+t-1}}x_{i-k,j}z_{i,j+t}		~\hbox{for $k>1$}
	\end{equation*}
	\begin{equation*}
		z_{i,j+t}x_{i-1,j}-q^{(-1)^{d+t-1}-(-1)^{t} }x_{i-1,j}z_{i,j+t}
	\end{equation*}
	\begin{equation*}
		z_{i,j+t}x_{i+k-1,j}-q^{(-1)^{dk+t-1}}x_{i+k-1,j}z_{i,j+t}
	\end{equation*}
	\begin{equation*}
		x_{i,j+t}z_{i+k,j}-q^{(-1)^{dk+t}}z_{i+k,j}x_{i,j+t}
	\end{equation*}
	\begin{equation*}
		x_{i,j+t}z_{i,j}-q^{(-1)^{d+t}+(-1)^t}z_{i,j}x_{i,j+t}
	\end{equation*}
	\begin{equation*}
		x_{i,j+t}z_{i-k+1,j}-q^{(-1)^{dk+t}}z_{i-k+1,j}x_{i,j+t} ~\hbox{ for $k>1$}
	\end{equation*}
\end{itemize}

\smallskip
\noindent Case $r>1$:
\begin{itemize}
	\item[(1)] the relations for $z_{i,j}$ and $z_{i',j}$, and for $z_{i,j}$ and $x_{i',j}$ ($i,i'\in\mathbb{Z}$ and $j=0,1,\ldots,d-1$):
	\begin{equation}\label{relation:projective-by-simple-and-projective-$r>1$}
		x_{i-1,j}z_{i,j}-z_{i,j}x_{i-1,j}-z_{i-1,j}
	\end{equation}
		\begin{equation*}
		x_{i-kr,j}z_{i,j}-q^{(-1)^{dk}}z_{i,j}x_{i-kr,j} 
	\end{equation*}
	\begin{equation*}
		x_{i+kr-1,j}z_{i,j}-q^{(-1)^{dk}}z_{i,j}x_{i+kr-1,j} 
	\end{equation*}
	\begin{equation*}
		x_{i',j}z_{i,j}-q^{\delta_{i,i'}}z_{i,j}x_{i',j}~\hbox{ for other $i,i'$} 
	\end{equation*}
	\begin{equation*}
		z_{i,j}z_{i',j}-q^{-1-\sum\limits_{p=1}^{k}(-1)^{dp} }z_{i',j}z_{i,j}~  (kr\leq i-i'<(k+1)r)
	\end{equation*}
	\begin{equation*}
		z_{i,j}z_{i',j}-q^{1+\sum\limits_{p=1}^{k}(-1)^{dp} }z_{i',j}z_{i,j}~  (-(k+1)r< i-i'\leq -kr)
	\end{equation*}
	\begin{equation*}
		z_{i,j}z_{i',j}-q^{\operatorname{sgn}(i'-i)}z_{i',j}z_{i,j}~\hbox{ for other $i,i'$}
	\end{equation*}

	\item[(2)] the relations for $z_{i,j+1}$ and $z_{i',j}$, for $z_{i,j+1}$ and $x_{i',j}$, and for $x_{i',j+1}$ and $z_{i,j}$ ($i,i'\in\mathbb{Z}$ and $j=0,1,\ldots,d-1$):
	\begin{equation*}
		z_{i,j+1}z_{i',j}-q^{\sum\limits_{p=1}^{k}(-1)^{dp}  }z_{i',j}z_{i,j+1}-\Phi_{i',i-1}(x_{\alpha,j})_{\alpha\in\mathbb{Z}}+\delta_{d,2}\delta_{i,i'+r}\frac{q^{(-1)^d}}{q-1}~(kr\leq i-i'<(k+1)r)
	\end{equation*}
	\begin{equation}\label{relation:simple-by-projectives-$r>1$}
		z_{i,j+1}z_{i',j}-z_{i',j}z_{i,j+1}-\Phi_{i',i-1}(x_{\alpha,j})_{\alpha\in\mathbb{Z}}+\delta_{d,2}\Phi_{i,i'+r-1}(x_{\alpha,j+1})_{\alpha\in\mathbb{Z}}~(0< i-i'<r)
	\end{equation}
	\begin{equation*}
		z_{i,j+1}z_{i,j}-q^{-1}z_{i,j}z_{i,j+1}-\frac{1}{q-1}+\delta_{d,2}q^{-1}\Phi_{i,i+r-1}(x_{\alpha,j+1})_{\alpha\in\mathbb{Z}}
	\end{equation*}
	\begin{equation*}
		z_{i,j+1}z_{i',j}-q^{-1}z_{i',j}z_{i,j+1}+\delta_{d,2}q^{-1}\Phi_{i,i'+r-1}(x_{\alpha,j+1})_{\alpha\in\mathbb{Z}}~(-r< i-i'<0)
	\end{equation*}
	\begin{equation*}
		z_{i,j+1}z_{i',j}-q^{-1+\sum\limits_{p=1}^{k}(-1)^{dp+1}}z_{i',j}z_{i,j+1}+\delta_{d,2}q^{-1+\sum\limits_{p=1}^{k}(-1)^{dp+1}}\Phi_{i,i'+r-1}(x_{\alpha,j+1})_{\alpha\in\mathbb{Z}}~(-(k+1)r<i-i'\leq -kr)
	\end{equation*}
	\begin{equation*}
		z_{i,j+1}x_{i-kr,j}-q^{(-1)^{dk}}x_{i-kr,j}z_{i,j+1}
	\end{equation*}
	\begin{equation*}
		z_{i,j+1}x_{i,j}-x_{i,j}z_{i,j+1}-z_{i+1,j+1}
	\end{equation*}
	\begin{equation*}
		z_{i,j+1}x_{i+kr-1,j}-q^{(-1)^{dk}}x_{i+kr-1,j}z_{i,j+1}~\hbox{ for $k\geq 0$} 
	\end{equation*}
	\begin{equation*}
		z_{i,j+1}x_{i',j}-x_{i',j}z_{i,j+1}~\hbox{ for other $i,i'$} 
	\end{equation*}
	\begin{equation*}
		x_{i-kr,j+1}z_{i,j}-q^{(-1)^{dk+1}}z_{i,j}x_{i-kr,j+1}~\hbox{ for $k\geq 0$}  
	\end{equation*}
	\begin{equation*}
		x_{i+r,j+1}z_{i,j}-z_{i,j}x_{i+r,j+1}+\delta_{d,2}z_{i+1,j} 
	\end{equation*}
	\begin{equation*}
		x_{i+kr-1,j+1}z_{i,j}-q^{(-1)^{dk+1}}z_{i,j}x_{i+kr-1,j+1}
	\end{equation*}
		\begin{equation*}
		x_{i',j+1}z_{i,j}-             z_{i,j}x_{i',j+1}~\hbox{ for other $i,i'$} 
	\end{equation*}

	\item[(3)] the relations for $z_{i,j+t}$ and $z_{i',j}$, for $z_{i,j+t}$ and $x_{i',j}$, and for $x_{i',j+t}$ and $z_{i,j}$ (for $x_{i,j+t}$ and $x_{i',j}$ are the same as in the previous section) ($i,i'\in\mathbb{Z}$ and $j=0,1,\ldots,d-1$):
	\begin{equation*}
		z_{i,j+t}z_{i',j}-q^{\sum\limits_{p=1}^{k}(-1)^{dp+t-1}}z_{i',j}z_{i,j+t}~(kr\leq i-i'<(k+1)r)
	\end{equation*}
	\begin{equation*}
		z_{i,j+t}z_{i',j}-z_{i',j}z_{i,j+t}~(0< i-i'<r)
	\end{equation*}
	\begin{equation*}
		z_{i,j+t}z_{i',j}-q^{(-1)^t}z_{i',j}z_{i,j+t}~(-r< i-i'\leq 0)
	\end{equation*}
	\begin{equation*}
		z_{i,j+t}z_{i',j}-q^{(-1)^t+\sum\limits_{p=1}^{k}(-1)^{dp+t}}z_{i',j}z_{i,j+t}~(-(k+1)r<i-i'\leq -kr)
	\end{equation*}
	\begin{equation*}
		z_{i,j+t}x_{i-kr,j}-q^{(-1)^{dk+t-1}}x_{i-kr,j}z_{i,j+t}
	\end{equation*}
	\begin{equation*}
		z_{i,j+t}x_{i+kr-1,j}-q^{(-1)^{dk+t-1}}x_{i+kr-1,j}z_{i,j+t}~\hbox{ for $k\geq 0$}
	\end{equation*}
	\begin{equation*}
		z_{i,j+t}x_{i',j}-x_{i',j}z_{i,j+t}~\hbox{ for other $i,i'$} 
	\end{equation*}
	\begin{equation*}
		x_{i-kr,j+t}z_{i,j}-q^{(-1)^{dk+t}}z_{i,j}x_{i-kr,j+t}~\hbox {for $k\geq 0$}
	\end{equation*}
	\begin{equation*}
		x_{i+kr-1,j+t}z_{i,j}-q^{(-1)^{dk+t}}z_{i,j}x_{i+kr-1,j+t}
	\end{equation*}
	\begin{equation*}
		x_{i',j+t}z_{i,j}-z_{i,j}x_{i',j+t}~\hbox{ for other $i,i'$} 
	\end{equation*}
\end{itemize}

\begin{remark}
\label{rem:other-generators}
We choose $\{z_{i,j}\}_{(i\in \mathbb{Z},j\in 0,1,\ldots, d-1)}\cup \{x_{i,j}\}_{(i\in \mathbb{Z},j\in 0,1,\ldots, d-1)}$ as a set of generators for $\ch^+(0,1,d,r)$. Actually, when $d>2$, $\{z_{i,j}\}_{(i\in \mathbb{Z},j\in 0,1,\ldots, d-1)}$ is already a set of generators. Indeed, it follows by the relation \eqref{relation:simple-by-projectives-$r=1$} when $r=1$ and the relation \eqref{relation:simple-by-projectives-$r>1$} when $r>1$ (notice that $\Phi_{i,i}(x_{\alpha,j})_{\alpha\in\mathbb{Z}}=x_{i,j}$) that $x_{i,j}$ can be generated by $\{z_{i,j}\}_{(i\in \mathbb{Z},j\in 0,1,\ldots, d-1)}$. This is also true for the case $d=2$ and $r=1$, but it does not seem to be true for $d=2$ and $r>1$ unless we allow infinite sums. The advantage of choosing $\{z_{i,j}\}_{(i\in \mathbb{Z},j\in 0,1,\ldots, d-1)}\cup \{x_{i,j}\}_{(i\in \mathbb{Z},j\in 0,1,\ldots, d-1)}$ as generators is that the presentation of relations is much simpler. 
	
It follows by induction from the relations \eqref{relation:projective-by-simple-and-projective-$r=1$} and \eqref{relation:projective-by-simple-and-projective-$r>1$} that $\{z_{0,j}\}_{(j\in 0,1,\ldots, d-1)}\cup \{x_{i,j}\}_{(i\in \mathbb{Z},j\in 0,1,\ldots, d-1)}$ is also a set of generators for $\ch^+(0,1,d,r)$. This is exactly the set of generators which G. Bobi\'{n}ski and J. Schmude used in \cite{BobinskiSchmude20} for the case $1\leq d\leq r$.  
\end{remark}

Note that the relations in group (3) are empty for $d=2$ and $d=3$. The strategy of the proof of Theorem~\ref{thm-for-linear-$+$} is the same as that of Theorem~\ref{thm-for-zigzag}.
For $j=0,1,\ldots,d-1$, let $\ch_j(\cc^+)$ be the subspace of $\ch(\cc^+)$ spanned by $[\Sigma^j M]$, $M\in\rep^+(Q^l)$. Then it is a subalgebra of $\ch(\cc^+)$, by Proposition~\ref{prop:Hall-subalgebra-in-root-category} (a). Let $\ch^+_j$ be the subalgebra of $\ch^+$ generated by $x_{ij}$ and $z_{ij}$ ($i\in\mathbb{Z}$).

\begin{lemma}
	\label{lem:relations-in-the-first-group-linear-projectives}
	Let $j=0,1,\ldots,d-1$.
	The assignment $x_{i,j}\mapsto [\Sigma^j S_i]$, $z_{i,j}\mapsto [\Sigma^j P_i]$  ($i\in\mathbb{Z}$) defines an isomorphism $\ch^+_j\to\ch_j(\cc^+)$ of $\mathbb{Q}$-algebras.
\end{lemma}

This lemma verifies the relations in group (1). Its proof is similar to that of Lemma~\ref{lem:relations-in-the-first-group-zigzag}.
The relations in group (2) follow from Corollary~\ref{cor:commutator-for-neighbouring-pieces-projectives}, Corollary~\ref{cor:commutator-for-neighbouring-pieces-projective-simple}, Lemma~\ref{lem:Euler-form-linear} and Corollary~\ref{cor:antisymmetric-orbital-Euler-form-linear-projectives}. The relations in group (3) directly follow from Corollary~\ref{cor:commutator-for-pieces-faraway} and Lemma~\ref{lem:Euler-form-linear} and Corollary~\ref{cor:antisymmetric-orbital-Euler-form-linear-projectives}.  
Therefore, the assignment $x_{i,j}\mapsto [\Sigma^{j}S_i]$ ($i\in\mathbb{Z}$ and $j=0,1,\ldots,d-1$) defines a homomorphism $\ch^+\to\ch(\cc^+)$ of $\mathbb{Q}$-algebras, which we denote by $\varphi$. Moreover,
\[
\ch^+=\ch^+_0\otimes_{\mathbb{Q}}\ch^+_1\otimes_{\mathbb{Q}}\cdots\otimes_{\mathbb{Q}}\ch^+_{d-1}
~~\text{and}~~
\ch(\cc^+)=\ch_0(\cc^+)\otimes_\mathbb{Q}\ch_1(\cc^+)\otimes_\mathbb{Q}\cdots\otimes_\mathbb{Q}\ch_{d-1}(\cc^+)
\]
as $\mathbb{Q}$-vector spaces. Since $\varphi$ restricts to isomorphisms $\ch^+_j\to\ch_j(\cc^+)$ for all $j=0,1,\ldots,d-1$ by Lemma~\ref{lem:relations-in-the-first-group-linear-projectives}, it follows that $\varphi$ is an isomorphism.

\section{Derived Hall algebras of graded gentle one-cycle algebras}
\label{s:derive-Hall-algebra-of-graded-gentle-one-cycle-algebra}

In this section we describe the derived Hall algebra of a (finite-dimensional) graded gentle one-cycle algebra.

\medskip

Let $A$ be a graded gentle one-cycle algebra. If $A$ has finite global dimension, then the AG-invariant of $A$ is  $\{(p_1,p_1+d),(p_2,p_2-d)\}$ for some integers $p_1,p_2\geq 1$ and $d\geq 0$. If $A$ has infinite global dimension, then the AG-invariant of $A$ is $\{(q,q-d),(0,d)\}$ for some integers $q\geq 1$ and $d\in\mathbb{Z}$. See for example \cite[Lemma 6.2]{ChenYang25}. We have the following description of the perfect and finite-dimensional derived categories of a graded gentle one-cycle algebra by twisted root categories.

\begin{theorem}[{\cite[Theorem 6.6]{ChenYang25}}]
\label{thm:from-gentle-one-cycle-to-root-category}
Let $A$ be a graded gentle one-cycle algebra.
\begin{itemize}
\item[(a)] If the AG-invariant of $A$ is $\{(p_1,p_1+d),(p_2,p_2-d)\}$ for some integers $p_1,p_2,d\geq 1$, then there is a triangle equivalence $\per(A)\to\cd^b(\rep^b(Q_{p_1,p_2}))/\Sigma^d\sigma^{-1}$.
\item[(b)] If the AG-invariant of $A$ is $\{(q,q-d),(0,d)\}$ for some $q\geq 1$ and $d\neq 0$, then there is a triangle equivalence $\cd_{fd}(A)\to\cd^b(\rep^+(Q^l))/\Sigma^{|d|}\sigma^{-\mathrm{sgn}(d)q}$ which restricts to a triangle equivalence $\per(A)\to\cd^b(\rep^b(Q^l))/\Sigma^{|d|}\sigma^{-\mathrm{sgn}(d)q}$.
\end{itemize}
\end{theorem}

If follows from Theorem~\ref{thm:from-gentle-one-cycle-to-root-category} and \cite[Lemmas 4.3 and 4.7]{ChenYang25} that if $d\neq 0$, then the number of connected components of the Auslander--Reiten quiver of $\per(A)$ is $3|d|$, when $A$ has finite global dimension, or $|d|$, when $A$ has infinite global dimension.

Recall that we defined some $\mathbb{Q}$-algebras $\ch(p_1,p_2,d,r)$ in the beginning of Section~\ref{s:derived-Hall-of-zigzag}, $\ch(0,1,d,r)$ in the beginning of Section~\ref{ss:derived-Hall-of-linear-bounded} and $\ch^+(0,1,d,r)$ in the beginning of Section~\ref{ss:derived-Hall-of-linear-$+$}. The following theorem is an immediate consequence of Theorem~\ref{thm:from-gentle-one-cycle-to-root-category}, Theorem~\ref{thm-for-zigzag}, Theorem~\ref{thm-for-linear} and Theorem~\ref{thm-for-linear-$+$}.

\begin{theorem}
\label{thm:derived-hall-algebra}
Let $A$ be a graded gentle one-cycle algebra.
\begin{itemize}
\item[(a)] Assume the AG-invariant of $A$ is $\{(p_1,p_1+d),(p_2,p_2-d)\}$ for some integers $p_1,p_2\geq 1$ and $d\geq 2$.Then the derived Hall algebra $\ch(\per(A))$ is isomorphic to $\ch(p_1,p_2,d,-1)$.
\item[(b)] Assume the AG-invariant of $A$ is $\{(q,q-d),(0,d)\}$ for some $q\geq 1$ and $|d|\geq 2$. Then the derived Hall algebra $\ch(\per(A))$ is isomorphic to $\ch(0,1,|d|,\mathrm{sgn}(d)q)$.
\item[(c)] Assume the AG-invariant of $A$ is $\{(q,q-d),(0,d)\}$ for some $q\geq 1$ and $d\geq 2$. Then the derived Hall algebra $\ch(\cd_{fd}(A))$ is isomorphic to $\ch^+(0,1,d,q)$.
\end{itemize}
\end{theorem}

\begin{remark}
Let $A$ be a graded gentle one-cycle algebra. If its AG-invariant is $\{(p_1r,p_1r+d),(p_2r,p_2r-d)\}$ for some integers $p_1,p_2,r\geq 1$ and $d\geq 2$, then according to the proof of \cite[Theorem 6.7]{ChenYang25}, there are triangle equivalences $\per(A)\to\cd^b(\rep^b(Q_{p_1,p_2}))/\Sigma^d\sigma^{-r}\to\cd^b(\rep^b(Q_{p_2,p_1}))/\Sigma^d\sigma^{r}$. It then follows by Theorem~\ref{thm-for-zigzag} that there are isomorphisms $\ch(\per(A))\cong\ch(p_1,p_2,d,r)\cong\ch(p_2,p_1,d,-r)$. This gives different presentations of $\ch(\per(A))$.
\end{remark}

\begin{remark}
Let $A$ be a graded gentle one-cycle algebra. If the AG-invariant of $A$ is $\{(p_1,p_1+d),(p_2,p_2-d)\}$ for some integers $p_1,p_2\geq 1$ and $d\geq 1$, then $A$ is derived equivalent to an ungraded gentle one-cycle algebra if and only if $1\leq d\leq p_2$. If the AG-invariant of $A$ is $\{(q,q-d),(0,d)\}$, then $A$ is derived equivalent to an ungraded gentle one-cycle algebra if and only if $1\leq d\leq q$. See \cite[Remark 6.3 (c)]{ChenYang25}. In other words, $A$ is derived equivalent to an ungraded gentle one-cycle algebra if and only if all integers appearing in the AG-invariant of $A$ are non-negative. In this case and when $A$ is of infinite global dimension, Theorem~\ref{thm:derived-hall-algebra} (b) and (c) are given in \cite{BobinskiSchmude20} in terms of the set of generators $\{z_{0,j}\}_{(j\in 0,1,\ldots, d-1)}\cup \{x_{i,j}\}_{(i\in \mathbb{Z},j\in 0,1,\ldots, d-1)}$ (see Remark~\ref{rem:other-generators}).
\end{remark}

\begin{remark}
\label{rem:d=0}
The cases $|d|=0,1$ are exceptional. The case $|d|=1$ will be treated in a separate paper. 
Now assume $d=0$. If the AG-invariant of $A$ is $\{(p_1,p_1),(p_2,p_2)\}$, then $\per(A)$ is triangle equivalent to $\cd^b(\ca)$, the bounded derived category of finite-dimensional representations $\ca=\rep^b(Q)$ of $Q$, where $Q$ is a/any quiver of type $\tilde{\mathbb{A}}_{p_1,p_2}$; if the AG-invariant of $A$ is $\{(p_2,p_2),(0,0)\}$, then $\per(A)$ is triangle equivalent to $\cd^b(\ca)$, the bounded derived category of finite-dimensional nilpotent representations $\ca=\rep_{\mathrm{nil}}^b(Q)$ of $Q$, where $Q$ is the cyclic quiver with $p_2$ vertices. In both cases, the Auslander--Reiten quiver of $\per(A)$ has infinitely many connected components. In both cases, a set of generators and relations of $\per(A)$ can be obtained by using \cite[Section 7]{Toen06} together with the structure of the Ringel--Hall algebras of $\ca$ provided by \cite[Theorem 1.1]{SevenhantvandenBergh01}.
\end{remark}

\def\cprime{$'$}
\providecommand{\bysame}{\leavevmode\hbox to3em{\hrulefill}\thinspace}
\providecommand{\MR}{\relax\ifhmode\unskip\space\fi MR }
\providecommand{\MRhref}[2]{%
  \href{http://www.ams.org/mathscinet-getitem?mr=#1}{#2}
}
\providecommand{\href}[2]{#2}

\end{document}